\documentclass[reqno,tbtags,a4paper,12pt]{amsart}

\usepackage[in]{fullpage}
\usepackage{makecell}

\usepackage[matrix,arrow,curve]{xy}

\usepackage{amsmath,amssymb,amsthm,amsfonts,amscd,mathtools}
\usepackage{graphics,graphicx}

\usepackage{mathtools}

\newcommand{\Z}{\mathbb{Z}}

\newcommand{\R}{\mathbb{R}}

\newcommand{\RP}{\mathbb{RP}}
\newcommand{\CP}{\mathbb{CP}}
\newcommand{\HP}{\mathbb{HP}}

\newcommand{\OP}{\mathbb{OP}}
\newcommand{\F}{\mathbb{F}}
\renewcommand{\P}{\mathbb{P}}

\newcommand{\CH}{\mathcal{H}}

\newcommand{\rC}{\mathrm{C}}
\newcommand{\rS}{\mathrm{S}}
\newcommand{\rA}{\mathrm{A}}
\newcommand{\rD}{\mathrm{D}}
\newcommand{\rQ}{\mathrm{Q}}
\newcommand{\He}{\mathrm{He}}

\newcommand{\PSU}{\mathrm{PSU}}
\newcommand{\PSL}{\mathrm{PSL}}
\newcommand{\Sz}{\mathrm{Sz}}

\newcommand{\Sym}{\mathop{\mathrm{Sym}}\nolimits}

\newcommand{\Isom}{\mathop{\mathrm{Isom}}\nolimits}
\newcommand{\link}{\mathop{\mathrm{link}}\nolimits}

\newcommand{\pt}{\mathrm{pt}}

\newtheorem{theorem}{Theorem} [section]
\newtheorem{propos}[theorem] {Proposition}
\newtheorem{cor}[theorem] {Corollary}
\newtheorem{lem}[theorem]{Lemma}

\theoremstyle{definition}
\newtheorem{remark}[theorem]{Remark}
\newtheorem{defin}[theorem]{Definition}

\numberwithin{equation}{section}

\author{Alexander A. Gaifullin}

\address{Steklov Mathematical Institute of Russian Academy of Sciences, Moscow, Russia}
\address{Skolkovo Institute of Science and Technology, Moscow, Russia}
\address{Lomonosov Moscow State University, Moscow, Russia}
\address{Institute for the Information Transmission Problems of the Russian Academy of Sciences (Kharkevich Institute), Moscow, Russia}
\email{agaif@mi-ras.ru}

\thanks{}

\title{On possible symmetry groups of 27-vertex triangulations of manifolds like the octonionic projective plane}
\date{}

\keywords{Minimal triangulation, octonionic projective plane, manifold like a projective plane, K\"uhnel triangulation, vertex-transitive triangulation, combinatorial manifold, transformation group, Smith theory, fixed point set, symmetry group}

\subjclass[2020]{57Q15, 57Q70, 05E45, 55M35}

\begin{document}
\begin{abstract}
In 1987 Brehm and K\"uhnel showed that any triangulation of a $d$-manifold (without boundary) that is not homeomorphic to the sphere has at least $3d/2+3$ vertices. Moreover, triangulations with exactly $3d/2+3$ vertices may exist only for `manifolds like projective planes', which can have dimensions~$2$, $4$, $8$, and~$16$ only. There is  a $6$-vertex triangulation of the real projective plane~$\RP^2$, a $9$-vertex triangulation of the complex projective plane~$\CP^2$, and $15$-vertex triangulations of the quaternionic projective plane~$\HP^2$. Recently, the author has constructed first examples of $27$-vertex triangulations of manifolds like the octonionic projective plane~$\OP^2$. The four most symmetrical of them have symmetry group $\rC_3^3\rtimes \rC_{13}$ of order~$351$. These triangulations were constructed using a computer program after the symmetry group was guessed. However, it remained unclear why exactly this group is realized as the symmetry group and whether $27$-vertex triangulations of manifolds like~$\OP^2$ exist with other (possibly larger) symmetry groups.  In this paper we find strong restrictions on symmetry groups of such $27$-vertex triangulations. Namely, we present a list of $26$ subgroups of~$\rS_{27}$ containing all possible symmetry groups of $27$-vertex triangulations of manifolds like the octonionic projective plane. (We do not know whether all these subgroups can be realized as symmetry groups.) The group $\rC_3^3\rtimes \rC_{13}$ is the largest group in this list, and the orders of all other groups do not exceed~$52$. A key role in our approach is played by the use of Smith and Bredon's results on the topology of fixed point sets of finite transformation groups.
\end{abstract}

\maketitle

\section{Introduction}\label{section_intro}

Throughout the paper, we denote by~$\rC_n$, $\rD_n$, $\rS_n$, and~$\rA_n$ the cyclic group of order~$n$, the dihedral group of order~$2n$, the symmetric group of degree~$n$, and the alternating group of degree~$n$, respectively.

In 1987 Brehm and K\"uhnel~\cite{BrKu87} obtained the following result.

\begin{theorem}[Brehm--K\"uhnel]\label{thm_BK}
 Suppose that $K$ is a combinatorial $d$-manifold with $n$ vertices.
 \begin{enumerate}
  \item If $n<3d/2+3$, then $K$ is PL homeomorphic to the standard sphere~$S^d$.
  \item If $n=3d/2+3$, then either $K$  is PL homeomorphic to~$S^d$ or $d\in\{2,4,8,16\}$ and $K$ is a `manifold like a projective plane', that is, $K$ admits a PL Morse function with exactly $3$ critical points.
 \end{enumerate}
\end{theorem}

We are not going to give a definition of a PL Morse function here. In the present paper we will need only the following properties of manifolds~$M$ like projective planes, see~\cite{EeKu62}:
\begin{itemize}
 \item If $d=2$, then $M$ is homeomorphic (and hence PL homeomorphic) to the real projective plane~$\RP^2$.
 \item If $d=4$, then $M$ is homeomorphic to the complex projective plane~$\CP^2$. (Nevertheless, it is not known whether it is true that $M$ is always PL homeomorphic to~$\CP^2$.)
 \item If $d\in\{4,8,16\}$, then
 $$
 H^*(M;\Z)\cong \Z[a]/(a^3),\qquad \deg a =d/2.
 $$
 \item In each of the dimensions $8$ and $16$, there is an infinite series of pairwise non-homeomorphic PL manifolds like projective planes. A complete classification of such manifolds was obtained by Kramer~\cite{Kra03}.
\end{itemize}

In view of Theorem~\ref{thm_BK}, for $d=2,4,8,16$, it is interesting to study combinatorial $d$-manifolds with exactly $3d/2+3$ vertices that are not homeomorphic to the sphere, and their symmetry groups.  Known examples of such combinatorial manifolds are as follows.
\begin{description}
 \item[$\boldsymbol{d=2}$] A $6$-vertex triangulation~$\RP^2_6$ is the quotient of the boundary of regular icosahedron by the antipodal involution. It is easy to show that up to isomorphism $\RP^2_6$ is the only $6$-vertex combinatorial $2$-manifold that is not homeomorphic to~$S^2$.
 The symmetry group of~$\RP^2_6$ is isomorphic to~$\rA_5$.

 \item[$\boldsymbol{d=4}$] A $9$-vertex triangulation~$\CP^2_9$ was constructed by K\"uhnel, see~\cite{KuBa83}. From the results of the papers~\cite{KuLa83}, \cite{BrKu87}, and~\cite{ArMa91} it follows that, up to isomorphism, $\CP^2_9$ is the only $9$-vertex combinatorial $4$-manifold that is not homeomorphic to~$S^4$. The symmetry group of~$\CP^2_9$ has order~$54$ and is isomorphic to a semi-direct product $\He_3\rtimes\rC_2$, where $\He_3$ is the Heisenberg group of order~$27$, see Subsection~\ref{subsection_K} for more detail.

 \item[$\boldsymbol{d=8}$] Brehm and K\"uhnel~\cite{BrKu92} constructed three $15$-vertex combinatorial $8$-mani\-folds like the quaternionic projective plane with symmetry groups~$\rA_5$, $\rA_4$, and~$\rS_3$, respectively. It was not known for a long time whether these combinatorial manifolds are actually homeomorphic to~$\HP^2$, until Gorodkov~\cite{Gor16,Gor19} proved that they are.

 \item[$\boldsymbol{d=16}$] Until recently, no example of a $27$-vertex combinatorial $16$-manifold like the octonionic projective plane was known. Recently, the  author~\cite{Gai22} has built a huge number (more than $10^{103}$) of such combinatorial manifolds. Four of them have symmetry group~$G_{351}=\rC_3^3\rtimes\rC_{13}$, and the symmetry groups of the others are subgroups of~$G_{351}$.
\end{description}

For interesting connections of $(3d/2+3)$-vertex triangulations of $d$-manifolds like projective planes with Freudenthal magic square and Severi varieties, see~\cite{ChMa13}.

The four $27$-vertex combinatorial $16$-manifolds like the octonionic projective plane with the symmetry group~$G_{351}$ were found by a special computer program, see~\cite{Gai-prog}. This program lists all $27$-vertex combinatorial $16$-manifold like the octonionic projective plane with a given symmetry group~$G$. It works rather well for groups~$G$ with $|G|\gtrsim 250$. However, its running time grows uncontrollably for groups of smaller orders. The fact that the program found examples of combinatorial manifolds for $G=G_{351}$ and did not find them for a number of other groups~$G$ of close orders had no reasonable explanation. In the present work, we clarify this issue. Our main result is as follows.

\begin{theorem}\label{thm_main}
 Suppose that $K$ is a $27$-vertex combinatorial $16$-manifold that is not homeomorphic to the sphere~$S^{16}$. Then the symmetry group $\Sym(K)$ is conjugate to one of the $26$ subgroups of~$\rS_{27}$ listed in Table~\ref{table_main}.
\end{theorem}

\begin{table}
\caption{Possible symmetry groups of $27$-vertex triangulations of $16$-manifolds like the octonionic projective plane}\label{table_main}
\begin{tabular}{|r|c|c|c|c|}
\hline
\textbf{No.} & \textbf{Group} & \textbf{Order} & \textbf{Action on vertices} & \textbf{Known}\\
& & & \textbf{(orbit lengths)} & \textbf{examples}\\
\hline
1 & $\rC_3^3\rtimes \rC_{13}$ & 351 & transitive & 4\\
2 & $\rC_3^3$ & 27 & regular transitive & 630\\
3 & $\He_3=3_+^{1+2}$ & 27 & regular transitive & no \\
4 & $3_-^{1+2}$ & 27 & regular transitive & no \\
\hline
5 & $\rC_{13}\rtimes\rC_4{}^{*}$ & 52 & 13, 13, 1 & no\\
6 & $\rC_3^2\rtimes\rC_4{}^{\dagger}$ & 36 & 9, 9, 9 & no \\
7 & $\rD_{13}$ & 26 & 13, 13, 1 & no \\
8 & $\rD_{11}$ & 22 & 11, 11, 2, 2, 1 & no \\
9 & $\rC_3\rtimes\rS_3$ & 18 & 9, 9, 9 & no \\
10 & $\rD_9$ & 18 & 9, 9, 9 & no \\
11 & $\rC_{13}$ & 13 & 13, 13, 1 & $>1.3\cdot 10^8$ \\
12 & $\rA_4$ & 12 & 12, 12, 3 & no\\
13 & $\rC_{11}$ & 11 & 11, 11, 1, 1, 1, 1, 1 & no \\
14 & $\rC_3^2$  & 9 & 9, 9, 9 & $> 1.8\cdot 10^{11}$\\
15 & $\rC_9$  & 9 & 9, 9, 9 & no\\
16 & $\rC_2^3$ & 8 & 8, 8, 8, 1, 1, 1 & no \\
17 & $\rD_4$ & 8 & 8, 8, 8, 1, 1, 1 & no \\
18 & $\rD_4$ & 8 & 8, 8, 4, 4, 2, 1$\lefteqn{^{\ddagger}}$ & no \\
19 & $\rQ_8$ & 8 & 8, 8, 8, 1, 1, 1 & no \\
20 & $\rS_3$ & 6 & 6, 6, 6, 3, 3, 3 & no \\
21 & $\rC_2^2$ & 4 & 4, 4, 4, 4, 4, 4, 1, 1, 1\phantom{$^{\S}$} & no \\
22 & $\rC_2^2$ & 4 & 4, 4, 4, 4, 4, 2, 2, 2, 1$^{\S}$ & no \\
23 & $\rC_4$ & 4 & 4, 4, 4, 4, 4, 4, 1, 1, 1\phantom{$^{\S}$} & no \\
24 & $\rC_3$ & 3 & free & $> 1.8\cdot 10^{34}$ \\
25 & $\rC_2$ & 2 & with 3 fixed points & no \\
26 & 1 & 1 & & $> 1.3\cdot 10^{103}$ \\
\hline
\multicolumn{5}{l}{\footnotesize{$^*$ with $\rC_4$ acting faithfully on~$\rC_{13}$}}\\
\multicolumn{5}{l}{\footnotesize{$^{\dagger}$ with $\rC_4$ acting faithfully on~$\rC_{3}^2$}}\\
\multicolumn{5}{l}{\footnotesize{$^{\ddagger}$ the stabilizers of points in the two different orbits of length $4$ are not   }} \\[-3pt]
\multicolumn{5}{l}{\footnotesize{\phantom{$^{\ddagger}$} conjugate to each other and  none of them is the centre of~$\rD_4$; }}\\[-3pt]
\multicolumn{5}{l}{\footnotesize{\phantom{$^{\ddagger}$} the stabilizer  of a point in  the orbit of length~$2$ is the subgroup $\rC_4\subset\rD_4$}}\\
\multicolumn{5}{l}{\footnotesize{$^{\S}$ the stabilizers of points in the three orbits of length $2$ are the three}} \\[-3pt]
\multicolumn{5}{l}{\footnotesize{\phantom{$^{\S}$} pairwise different subgroups of~$\rC_2^2$ isomorphic to~$\rC_2$}}
\end{tabular}
\end{table}

\begin{remark}
 Let us make several comments concerning Table~\ref{table_main}:
 \begin{itemize}
  \item Writing $N\rtimes H$, we always mean a semidirect product that is not direct. The semidirect products $\rC_3^3\rtimes \rC_{13}$ and $\rC_3\rtimes\rS_3$ are unique up to isomorphism. The semidirect products $\rC_{13}\rtimes \rC_4$ and $\rC_3^2\rtimes\rC_4$ in rows~5 and~6, respectively, are taken with respect to faithful actions of~$\rC_4$, which are again unique up to isomorphism.
  \item It is easy to check that, for each of the listed groups, except for~$\rD_4$ and~$\rC_2^2$ in rows~18 and~22, respectively, there is a unique up to isomorphism action on a $27$-element set with the prescribed orbit lengths. In rows~18 and~22, the table contains special comments on which of the actions is taken. So each row in the table determines a subgroup of~$\rS_{27}$ up to conjugation.
  \item $\He_3=3^{1+2}_+$ and $3^{1+2}_-$ are the two \textit{extraspecial groups} of order~$27$, i.\,e.,  the groups~$G$ with centre~$Z$ isomorphic to~$\rC_3$ and the quotient $G/Z$ isomorphic to~$\rC_3^2$, where the first one, the Heisenberg group~$\He_3$, has exponent~$3$, while the second one contains an element of order~$9$.
  \item $\rQ_8$ is the quaternion group.
 \end{itemize}
\end{remark}

\begin{remark}
 Up to now, examples of $27$-vertex combinatorial $16$-manifold like the octonionic projective plane are known only for $6$ of the $26$ potentially possible symmetry groups from Table~\ref{table_main}, namely, for the symmetry groups $G_{351}=\rC_3^3\rtimes \rC_{13}$, $\rC_3^3$, $\rC_{13}$, $\rC_3^2$, $\rC_3$, and~$1$. All these examples were constructed by the author in~\cite{Gai22}; the numbers of such examples are indicated in the last column of Table~\ref{table_main}. For the other $20$~groups, the question on the existence of a $27$-vertex combinatorial $16$-manifold like the octonionic projective plane with such symmetry group is open. The group~$G_{351}$ is the only symmetry group for which we know a complete list of all (the four) $27$-vertex combinatorial $16$-manifold like the octonionic projective plane. For the other $5$ symmetry groups, it is unknown whether the examples obtained in~\cite{Gai22} exhaust all possible combinatorial manifolds. Note also that it is still unknown whether the combinatorial manifolds constructed in~\cite{Gai22} are homeomorphic to~$\OP^2$ or they are other manifolds like the octonionic projective plane.
\end{remark}

\begin{remark}
A partial explanation for why the group~$G_{351}$ arises as the symmetry group of a $27$-vertex combinatorial $16$-manifold like the octonionic projective plane is that $G_{351}$ can be nicely realized as a subgroup of the isometry group~$\Isom(\OP^2)$, where $\OP^2$ is endowed with the Fubini--Study metric, see~\cite[Remark~1.7]{Gai22} for more detail. Namely, the group~$\Isom(\OP^2)$ is isomorphic to the exceptional simply connected compact Lie group~$F_4$ and Alekseevskii~\cite{Ale74} constructed an important subgroup of~$F_4$ isomorphic to~$\rC_3^3\rtimes\mathrm{SL}(3,\F_3)$, which contains~$G_{351}$, since $\mathrm{SL}(3,\F_3)$ contains an element of order~$13$. On the other hand, the group~$F_4$ has many different finite subgroups (see~\cite{CoWa97}) and, as Theorem~\ref{thm_main} shows, not all of them are realized as symmetry groups of $27$-vertex triangulations of manifolds like the octonionic projective plane. So the connection between the symmetry groups of such triangulations and finite subgroups of~$F_4$ is yet to be clarified.
\end{remark}

All of the above can be extended to the case of $\Z$-homology manifolds, see Definition~\ref{defin_hm} below. We start with the following analog of the Brehm--K\"uhnel theorem.

\begin{theorem}[Novik~\cite{Nov98}]\label{thm_Novik}
 Suppose that $K$ is a $\Z$-homology $d$-manifold with $n$ vertices.
 \begin{enumerate}
  \item If $n<3d/2+3$, then $H_*(K;\Z)\cong H_*(S^d;\Z)$.
  \item If $n=3d/2+3$, then either $H_*(K;\Z)\cong H_*(S^d;\Z)$ or  the following assertions hold:
 \begin{itemize}
  \item $d\in\{2,4,8,16\}$,
  \item if $d=2$, then $K\cong\RP^2_6$,
  \item if $d\in\{4,8,16\}$, then
  \begin{equation}\label{eq_iso_tp}
  H^*(K;\Z)\cong\Z[a]/(a^3),\qquad \deg a =d/2.
  \end{equation}
 \end{itemize}
 \end{enumerate}
\end{theorem}

\begin{remark}
 This theorem is not stated explicitly in~\cite{Nov98}. Let us explain how to extract it from there. First, Lemma~5.5 in~\cite{Nov98} impies assertion~(1) of Theorem~\ref{thm_Novik}, and besides implies that in the case of even~$d=2m$ and $n=3m+3$ one has
\begin{equation}\label{eq_beta}
 \beta_1=\cdots=\beta_{m-1}=0,
\end{equation}
 where $\beta_i$ are the Betti numbers of~$K$ with coefficients in any field of characteristic two or in any other field, provided that, in addition, $K$ is orientable. Since any $\Z$-homology $2$-manifold is a combinatorial manifold, in the case $d=2$ the required assertion follows immediately from the uniqueness of a $6$-vertex triangulation of a $2$-manifold not homeomorphic to~$S^2$. So we may assume that $d\ge 4$. Then by~\eqref{eq_beta} we have that $\beta_1=0$ with coefficients in the field~$\F_2$. Hence, the fundamental group $\pi_1(K)$ admits no nontrivial homomorphisms to~$\F_2$. It follows that $K$ is orientable and so~\eqref{eq_beta} holds with coefficients in any field. Then the Poincar\'e duality implies that $H_0(K;\Z)\cong H_{2m}(K;\Z)\cong\Z$, the group $H_m(K;\Z)$ is free abelian, and $H_i(K;\Z)=0$ for $i\notin\{0,m,2m\}$. Now, Theorem~5.6 in~\cite{Nov98} implies that $\beta_m\le 1$, so the group $H_m(K;\Z)$ is either trivial or isomorphic to~$\Z$. In the former case we have that $H_*(K;\Z)\cong H_*(S^d;\Z)$, and in the latter case, using the Poincar\'e duality, we arrive at isomorphism~\eqref{eq_iso_tp}. Finally, it is a well-known corollary of Adams' Hopf invariant one theorem (see~\cite{Ada60}) that such isomorphism may exist only for $d\in \{4,8,16\}$.
 \end{remark}

 A generalization of Theorem~\ref{thm_main} for homology manifolds is as follows.

 {\sloppy
 \begin{theorem}\label{thm_main_hom}
 Suppose that $K$ is a $27$-vertex $\Z$-homology $16$-manifold such that the graded group~$H_*(K;\Z)$ is not isomorphic to~$H_*(S^{16};\Z)$. Then the symmetry group $\Sym(K)$ is conjugate to one of the $26$ subgroups of~$\rS_{27}$ listed in Table~\ref{table_main}.
\end{theorem}
}

Certainly, Theorem~\ref{thm_main} follows from Theorem~\ref{thm_main_hom}. A natural conjecture (cf.~\cite[\S\,20, Conjecture~0]{ArMa91}) is that any $(3d/2+3)$-vertex $\Z$-homology $d$-manifold is a combinatorial manifold. This conjecture is proved for $d=2$ and $d=4$, see~\cite[\S\S\,19,\,21]{ArMa91}, but is completely open for $d=8$ and $d=16$. So it is not clear whether Theorem~\ref{thm_main_hom} is really stronger than Theorem~\ref{thm_main}.

Our approach to studying the symmetry groups~$\Sym(K)$ (and thus to proving Theorem~\ref{thm_main_hom}) is to consider elements $g\in\Sym(K)$ of different orders and study the topology and combinatorics of the corresponding fixed point sets~$K^{\langle g\rangle}$. (Hereafter we denote by~$\langle g\rangle$ the cyclic subgroup generated by~$g$.) We obtain results about the fixed point sets using two sets of ideas:
\begin{itemize}
 \item the Smith--Bredon cohomological theory of finite transformation groups,
 \item combinatorial results due to Bagchi and Datta on pseudomanifolds with complementarity.
\end{itemize}
Combining these two approaches, we obtain a series of conditions on orders of elements and subgroups of~$\Sym(K)$, see Proposition~\ref{propos_prop}. The next part of our proof is purely group theoretic: we prove that any group satisfying the obtained conditions is in Table~\ref{table_main}, see Proposition~\ref{propos_group}. More precisely, on  this stage we obtain one extra possibility for~$\Sym(G)$, namely, the group~$\PSU(3,2)$ of order~$72$ acting with three orbits of length~$9$. Then this extra possibility is excluded by a more detailed study of the fixed point complexes for subgroups isomorphic to~$\rC_4$.

All our proofs in this paper are not computer assisted, except for the proof of the fact that the symmetry group of a $15$-vertex $\F_2$-homology $8$-manifold that satisfies complimentarity cannot contain an element of order~$4$ (see Proposition~\ref{propos_no_4}). The proof of this fact uses a modification of the program, which was created by the author for finding $27$-vertex triangulations of $16$-manifolds like the octonionic projective plane (see~\cite{Gai-prog}).

\begin{remark}
It is natural to ask to what extent the results and methods of this work transfer to the case of $15$-vertex triangulations of~$\HP^2$ or manifolds like the quaternion projective plane.  One would expect that the conditions on the symmetry group in this case would be even more restrictive. Nevertheless, due to the difference in some number-theoretic properties of the numbers~$15$ and~$27$, it turned out that this is not the case. An attempt to apply the methods of this work in the $8$-dimensional case led to the discovery of a lot of new $15$-vertex triangulations of~$\HP^2$ with various symmetry groups and a partial classification of them, see~\cite{Gai23}.
\end{remark}

The present paper is organized as follows. Sections~\ref{section_simpcomp}, \ref{section_transform_groups}, and~\ref{section_finite_groups} provide preliminary information on simplicial complexes, transformation groups, and finite groups, respectively. Section~\ref{section_scheme} contains the scheme of the proof of Theorem~\ref{thm_main_hom}. In this section we formulate Proposition~\ref{propos_prop} (properties of $\Sym(K)$), Proposition~\ref{propos_group} (list of groups satisfying these properties), and Proposition~\ref{propos_no_PSU} (excluding the group~$\PSU(3,2)$), which together imply Theorem~\ref{thm_main_hom}. In Section~\ref{section_no_4}, we obtain an auxiliary result on the absence of elements of order~$4$ in the symmetry group of a $15$-vertex $\F_2$-homology $8$-manifold that satisfies complimentarity. Further, we prove Propositions~\ref{propos_prop}, \ref{propos_group}, and~\ref{propos_no_PSU} in Sections~\ref{section_prop}, \ref{section_group}, and~\ref{section_no_PSU}, respectively.

\smallskip

The author is grateful to Denis Gorodkov, Vasilii Rozhdestvenskii, and Constantin Shramov for useful discussions. The author would like to especially thank Andrey Vasil'ev, who pointed out a way to radically simplify the proof of Proposition~\ref{propos_group} by using the results of Higman and Suzuki on groups in which every element has prime power order.

\section{Preliminaries on simplicial complexes}\label{section_simpcomp}

\subsection{Simplicial complexes and group actions}

\begin{defin}
An (\textit{abstract}) \textit{simplicial complex} on vertex set~$V$ is a set~$K$ of finite subsets of~$V$ such that
\begin{itemize}
 \item $\varnothing\in K$,
 \item if $\sigma\in K$ and $\rho\subset\sigma$, then $\rho\in K$.
\end{itemize}
The \textit{dimension} of a simplex~$\sigma$ is the cardinality of~$\sigma$ minus one.
\end{defin}

We always assume that a simplicial complex has no \textit{ghost vertices}, that is, all one-element subsets of~$V$ are simplices of~$K$. In this paper we will work with finite simplicial complexes only. Throughout the paper, we denote by $\chi(K)$ the Euler characteristic of~$K$.
We denote by~$\Delta^k$ the standard $k$-dimensional simplex and by~$\partial\Delta^k$ the boundary of it.

The \textit{geometric realization}~$|K|$ is the set of all linear combinations
$$
x=\sum_{v\in V}x_{v}v
$$
such that
\begin{itemize}
 \item $x_v\in\R_{\ge 0}$,
 \item $\sum_{v\in V}x_v=1$,
 \item the set of all~$v$ with $x_v\ne 0$ is a simplex of~$K$.
\end{itemize}
The set $|K|$ is always endowed with the topology of the CW complex. For a simplex~$\sigma\in K$, the point
$$
b(\sigma)=\frac{1}{|\sigma|}\sum_{v\in \sigma}v
$$
is called the \textit{barycentre} of~$\sigma$.

We say that a finite group~$G$ \textit{acts simplicially} on~$K$ if $G$ acts on the vertex set~$V$ so that every element of~$G$ takes simplices of~$K$ to simplices of~$K$ and non-simplices of~$K$ to non-simplices of~$K$. A simplicial action of~$G$ on~$K$ induces a piecewise linear action of~$G$ on the geometric realization~$|K|$. We denote by~$|K|^G$ the set of all $G$-fixed points in~$|K|$. The following proposition is standard.

\begin{propos}\label{propos_KG}
 Suppose that $K$ is a finite simplicial complex on vertex set~$V$ with a simplicial action of a finite group~$G$. Let $\sigma_1,\ldots,\sigma_s$ be all $G$-orbits in~$V$ that are simplices of~$K$. Then the fixed point set~$|K|^G$ coincides with the geometric relization of the simplicial complex~$K^G$ with vertices $b(\sigma_1),\ldots,b(\sigma_s)$ such that
 $\bigl\{b(\sigma_{i_1}),\ldots,b(\sigma_{i_t})\bigr\}\in K^G$ if and only if $\sigma_{i_1}\cup\cdots\cup\sigma_{i_t}\in K$.
\end{propos}

The simplicial complex~$K^G$ will be called the \textit{fixed point complex} of the action.

\subsection{Pseudomanifods, combinatorial manifolds, and homology manifolds}

\begin{defin}
 A finite simplicial complex~$K$ is called a \textit{weak $d$-pseudomanifold} if it satisfies the following two conditions:
 \begin{enumerate}
  \item every simplex of~$K$ is contained in a $d$-simplex of~$K$,
  \item every $(d-1)$-simplex of~$K$ is contained in exactly two $d$-simplices of~$K$.
  \end{enumerate}
  A weak $d$-pseudomanifold $K$ is called a \textit{$d$-pseudomanifold} if, in addition, it satisfies the condition
  \begin{itemize}
  \item[(3)] $K$ is \textit{strongly connected}, i.\,e., for any two $d$-simplices $\sigma,\tau\in K$, there exists a sequence of $d$-simplices $\sigma=\rho_1,\rho_2,\ldots,\rho_n=\tau$ such that $\dim(\rho_i\cap\rho_{i+1})=d-1$ for all~$i$.
  \end{itemize}
 A $d$-pseudomanifold is said to be \textit{orientable} if the $d$-simplices of it can be endowed with compatible orientations, which means that the orientations of~$\sigma$ and~$\tau$ should induce the opposite to each other orientations of~$\sigma\cap\tau$ whenever $\dim(\sigma\cap\tau)=d-1$.
\end{defin}

We will need the following easy proposition.

\begin{propos}\label{propos_<d}
 Suppose that a nontrivial finite group~$G$ acts faithfully and simplicially on a $d$-pseudomanifold~$K$. Then $\dim K^G$ is strictly smaller than~$d$.
\end{propos}

\begin{proof}
 Obviously, $\dim K^G\le d$. Assume that $\dim K^G=d$. Then there exists a $d$-simplex $\sigma\in K$ such that every element of~$G$ fixes every vertex of~$\sigma$. It is easy to see that if a $d$-simplex~$\sigma$ satisfies this property, then any $d$-simplex $\tau\in K$ such that $\dim(\sigma\cap\tau)=d-1$ also satisfies the same property. Since $K$ is strongly connected, it follows that every element of~$G$ fixes every vertex of~$K$ and hence the action of $G$ on~$K$ is not faithful.
\end{proof}

\begin{defin}
 The \textit{link} of a simplex~$\sigma$ of a simplicial complex~$K$ is the simplicial complex
 $$
 \link(\sigma,K)=\{\tau\in K\colon \sigma\cap\tau=\varnothing,\, \sigma\cup\tau\in K\}.
 $$
\end{defin}

\begin{defin}
 A finite simplicial complex~$K$ is called a \textit{combinatorial $d$-manifold} if
 \begin{enumerate}
 \item every simplex of~$K$ is contained in a $d$-simplex,
 \item for each simplex $\sigma\in K$ such that $\sigma\ne\varnothing$ and $\dim\sigma<d$, the simplicial complex $\link(\sigma,K)$ is PL homeomorphic to the standard $(d-\dim\sigma-1)$-sphere.
 \end{enumerate}
\end{defin}

\begin{defin}\label{defin_hm}
Suppose that $R$ is either the ring~$\Z$ or a finite field.
A finite simplicial complex~$K$ is called an $R$-\textit{homology} $d$-\textit{manifold} if
\begin{enumerate}
 \item every simplex of~$K$ is contained in a $d$-simplex,
 \item for each simplex $\sigma\in K$ such that $\sigma\ne\varnothing$ and $\dim\sigma<d$, there is an isomorphism of graded groups
 $$
 H_*\bigl(\link(\sigma,K); R\bigr)\cong H_*\bigl(S^{d-\dim\sigma-1};R\bigr).
 $$
\end{enumerate}
\end{defin}

The following assertions are straightforward:
\begin{itemize}
 \item Any combinatorial $d$-manifold is a $\Z$-homology $d$-manifold.
 \item Any $\Z$-homology $d$-manifold is a $\F_q$-homology $d$-manifold for any finite field~$\F_q$.
 \item If $R$ is either~$\Z$ or a finite field, then any connected $R$-homology $d$-manifold is a $d$-pseudomanifold.
\end{itemize}

\subsection{Complementarity}

The following property will play a key role in our considerations.

\begin{defin}
 We say that a finite simplicial complex~$K$ on vertex set~$V$ \textit{satisfies complementarity} if, for each subset $\sigma\subseteq V$, exactly one of the two subsets~$\sigma$ and~$V\setminus \sigma$ is a simplex of~$K$.
\end{defin}

The following theorem by Arnoux and Marin~\cite[\S 20]{ArMa91}, being combined with Theorem~\ref{thm_Novik},  implies  that all $(3d/2+3)$-vertex $\Z$-homology $d$-manifolds with $H_*(K;\F_2)\not\cong H_*(S^d;\F_2)$ satisfy complementarity.

\begin{theorem}[Arnoux--Marin]\label{thm_ArMa}
 Suppose that $K$ is a simplicial complex such that the cohomology ring~$H^*(K;\F_2)$ contains a subring isomorphic to $\F_2[a]/(a^3)$, where $m=\deg a$ is even. Then $K$ has at least $3m+3$ vertices. Moreover, if $K$ has exactly $3m+3$ vertices, then $K$ satisfies complimentarity.
\end{theorem}

The next theorem is due to Datta~\cite{Dat98}, except for the case of a $6$-pseudomanifold with $12$~vertices, which was excluded by Bagchi and Datta~\cite{BaDa04}.

\begin{theorem}[Datta, Bagchi--Datta]\label{thm_Datta}
Suppose that $K$ is a $d$-pseudomanifold that satisfies complementarity, where $d>0$. Then one of the following assertions hold:
\begin{itemize}
 \item $K\cong\RP^2_6$,
 \item $K\cong\CP^2_9$,
 \item $d\ge 7$ and the number of vertices of~$K$ is at least~$d+7$.
\end{itemize}
\end{theorem}

\begin{remark}
 Throughout the paper, we will also encounter another type of simplicial complexes that satisfy complementarity, namely, the disjoint unions~$\pt\sqcup\partial\Delta^k$, where $\pt$ is a point and~$\partial\Delta^k$ is the boundary of a $k$-dimensional simplex. It is not hard to show that they are the only disconnected simplicial complexes that satisfy complementarity. In particular, for $k=1$, we arrive at the disjoint union of three points.
\end{remark}

We will also need the following easy observation by Bagchi and Datta, see~\cite{BaDa04}.

\begin{propos}[Bagchi--Datta]\label{propos_BD}
Suppose that $K$ is a finite simplicial complex that satisfies complementarity. Then the Euler characteristic $\chi(K)$ is odd. Moreover, if the number of vertices of~$K$ is even, then $\chi(K)=1$.
\end{propos}

From the Poincar\'e duality it follows that the Euler characteristic of any odd-dimensio\-nal $\F_p$-homology manifold is equal to zero.

\begin{cor}\label{cor_no_odd_hom_mfld}
 No odd-dimensional $\F_p$-homology manifold satisfies complementarity.
\end{cor}

Complementarity is closely related with neighborliness.

\begin{defin}
 A simplicial complex~$K$ on vertex set~$V$ is said to be $k$-\textit{neighborly} if every $k$-element subset of~$V$ is a simplex of~$K$.
\end{defin}

The next proposition follows immediately from the definition.

\begin{propos}\label{propos_comp_neigh}
 Suppose that $K$ is an $n$-vertex $d$-dimensional simplicial complex that satisfies complementarity. Then $K$ is $(n-d-2)$-neighborly.
\end{propos}

Finally, we will need the following result on how the complementarity property is inherited when passing to the fixed point complex.

\begin{propos}\label{propos_complement}
 Suppose that $K$ is a finite simplicial complex that satisfies complementarity, $V$ is the vertex set of~$K$, and $G$ is a finite group acting simplicially on~$K$. Let $r$ be the number of $G$-orbits in~$V$.
 \begin{enumerate}
  \item If $r=1$, then $K^G$ is empty.
  \item If $r=2$, then $K^G$ is a point.
  \item If $r\ge 3$, then  either $K^G$ is an $(r-2)$-simplex or $K^G$ has $r$ vertices and satisfies complementarity.
 \end{enumerate}
 Moreover, $K^G$ is either empty or a simplex if and only if at least one of the $H$-orbits in~$V$ is not a simplex of~$K$.
\end{propos}

\begin{proof}
 Let $\sigma_1,\ldots,\sigma_r$ be all $G$-orbits in~$V$. First, assume that $r=1$. Since $\varnothing\in K$, by complementarity we have $\sigma_1=V\notin K$. So $K^G$ is empty  by Proposition~\ref{propos_KG}.

 Second, assume that $r=2$. By complementarity, exactly one of the two $G$-orbits~$\sigma_1$ and~$\sigma_2$ is a simplex of~$K$. So $K^G$ is a point.

Third, assume that $r\ge 3$. We have two cases.
\smallskip

\textsl{Case 1: One of $G$-orbits in~$V$, say~$\sigma_1$, is not a simplex of~$K$.} Then $K^G$ is a simplicial complex with $r-1$ vertices $b(\sigma_2),\ldots,b(\sigma_r)$. Since $\sigma_1\notin K$, by complementarity we have that $\sigma_2\cup\cdots\cup\sigma_r\in K$. So $K^G$ is the $(r-2)$-simplex with the vertices $b(\sigma_2),\ldots,b(\sigma_r)$.
\smallskip

\textsl{Case 2: All $G$-orbits in~$V$ are simplices of~$K$.} Then $K^G$ is a simplicial complex with $r$ vertices $b(\sigma_1),\ldots,b(\sigma_r)$. The complementarity property for~$K^G$ follows immediately from the complementarity property for~$K$.
\end{proof}

\subsection{K\"uhnel's $\CP^2_9$}\label{subsection_K}

In this subsection we recall an explicit description of K\"uhnel's triangulation~$\CP^2_9$ in terms of an affine plane over a field of three elements. This construction is due to Bagchi and Datta, see~\cite{BaDa94}.

Let $\mathcal{P}$ be an affine plane over the field~$\F_3$. Then $|\mathcal{P}|=9$. Fix a decomposition
$$
\mathcal{P}=\ell_0\cup\ell_1\cup\ell_2,
$$
where $\ell_0$, $\ell_1$, and~$\ell_2$ are mutually parallel lines, and fix a cyclic order of these three lines. We conveniently consider the indices~$0$, $1$, and~$2$ as elements of~$\F_3$. The three lines $\ell_0$, $\ell_1$, and~$\ell_2$ will be called \textit{special}, and all other lines in~$\mathcal{P}$ \textit{non-special}.

The affine plane~$\mathcal{P}$ will serve as the set of vertices of~$\CP^2_9$. The $4$-dimensional simplices of~$\CP^2_9$ are exactly the following $36$ five-element subsets of~$\mathcal{P}$:
\begin{enumerate}
 \item The $27$ subsets of the form $m_1\cup m_2$, where $\{m_1,m_2\}$ is a pair of intersecting non-special lines in~$\mathcal{P}$.
 \item The $9$ subsets of the form $\ell_t\cup(\ell_{t+1}\setminus\{v\})$, where $t\in\F_3$ and $v\in\ell_{t+1}$.
\end{enumerate}

The simplicial complex~$\CP^2_9$ consisting of these $36$ four-simplices and all their faces is a combinatorial PL manifold PL homeomorphic to~$\CP^2_9$. Besides, $\CP^2_9$ satisfies complementarity.

\begin{remark}
 In~\cite{BaDa94} the $9$ simplices of the second type are $(\ell_t\setminus\{u\})\cup\ell_{t+1}$, where $t\in\F_3$ and $u\in\ell_t$. So our convention on the cyclic order of the three special lines is opposite to the one used in~\cite{BaDa94}.
\end{remark}

It follows immediately from the construction that the simplicial complex~$\CP^2_9$ is invariant under all affine transformations of~$\mathcal{P}$ that take special lines to special lines and preserve the cyclic order of them. In fact, such affine transformations form the whole group~$\Sym(\CP^2_9)$, which has order~$54$. If we introduce affine coordinates~$(x,y)$ in~$\mathcal{P}$ so that the special lines are the lines~$\{y=0\}$, $\{y=1\}$, and~$\{y=2\}$ in this cyclic order, then the group~$\Sym(\CP^2_9)$ consists of all transformations of the form
\begin{equation}
 \label{eq_affine_transform}
\begin{pmatrix}
 x\\ y
\end{pmatrix}
\mapsto
\begin{pmatrix*}[r]
 \pm 1 & c\\
 0 & 1
\end{pmatrix*}
\begin{pmatrix}
 x\\ y
\end{pmatrix}+
\begin{pmatrix}
 a\\ b
\end{pmatrix},\qquad a,b,c\in\F_3,
\end{equation}
which can be rewritten as
\begin{equation*}
\begin{pmatrix}
 x\\ y\\1
\end{pmatrix}
\mapsto
\begin{pmatrix*}[r]
 \pm 1 & c & a\\
 0 & 1 & b\\
 0&0&1
\end{pmatrix*}
\begin{pmatrix}
 x\\ y \\ 1
\end{pmatrix}.
\end{equation*}
It follows that $\Sym(\CP^2_9)\cong\He_3\rtimes\rC_2$.

By Sylow theorems all elements of order~$2$ in~$\Sym(\CP^2_9)$ are conjugate to each other and hence are conjugate to the transformation
$$
(x,y)\mapsto(-x,y).
$$
The following result is obtained in~\cite[Section~7.2]{BrKu92}.

\begin{propos}[Brehm--K\"uhnel]\label{propos_BK}
 For any element $g\in\Sym(\CP^2_9)$ of order~$2$, the fixed point complex~$(\CP^2_9)^{\langle g\rangle}$ is isomorphic to~$\RP^2_6$.
\end{propos}

Also, we will need the following proposition.

\begin{propos}\label{propos_3subset}
 Suppose that $m$ is a $3$-element subset of~$\mathcal{P}$. Then the following two assertions are equivalent to each other:
 \begin{enumerate}
  \item all $4$-element subsets of~$\mathcal{P}$ that contain~$m$ are simplices of~$\CP^2_9$,
  \item $m$ is a non-special line.
 \end{enumerate}
\end{propos}

\begin{proof}
It can be easily checked that any $3$-element subset $m\subset\mathcal{P}$ can be taken by a transformation of the form~\eqref{eq_affine_transform} to one of the following $5$ subsets:
\begin{itemize}
 \item the non-special line $\{x=0\}$,
 \item the special line $\{y=0\}$,
 \item $\{(0,0),(0,1),(1,2)\}$,
 \item $\{(0,0),(1,0),(0,1)\}$,
 \item $\{(0,0),(1,0),(0,2)\}$,
\end{itemize}
so it is enough to check the statement of the proposition for each of these $5$ subsets, which can be done by a direct enumeration.
\end{proof}

\section{Preliminaries on transformation groups}\label{section_transform_groups}

Throughout this section,  $K$ is a finite simplicial complex and $p$ is a prime.

We need two classical results from theory of transformation groups. In the following theorem, assertions~(1) and~(2) are due to Smith~\cite[Theorems~1, 5 and~6]{Smi39} and assertion~(3) is due to Bredon~\cite[Theorem~7.7]{Bre60}, see also~\cite[Chapter~V]{Bor60} for more detail.

\begin{theorem}[Smith, Bredon]\label{thm_Smith}
 Suppose that $K$ is an $\F_p$-homology manifold, and $G$ is a finite $p$-group that acts faithfully and simplicially on~$K$. Then
 \begin{enumerate}
  \item  every connected component of~$K^{G}$ is an $\F_p$-homology manifold,
  \item if $K$ is orientable and $p$ is odd, then all connected component of~$K^{G}$ are orientable,
  \item if $p$ is odd, then $\dim K-\dim F$ is even for all connected components~$F$ of~$K^{G}$.
 \end{enumerate}
\end{theorem}

To formulate the second result we conveniently use the following standard  notation. We write $K\sim_p\P^h(m)$ if
\begin{equation}
H^*(K;\F_p)\cong\F_p[a]/(a^{h+1}),\qquad \deg a = m.
\end{equation}
 The following theorem is due to Bredon~\cite{Bre64}, see also~\cite[Theorem~4.1]{Bre68} and~\cite[Theorem~VII.3.1]{Bre72}.

\begin{theorem}[Bredon]\label{thm_Bredon}
 Suppose that $K\sim_p\P^h(m)$ and the group~$\rC_p$ acts simplicially on~$K$. Then either $K^{\rC_p}=\varnothing$ or
 $$
 K^{\rC_p}=F_1\sqcup\cdots\sqcup F_s
 $$
 for certain $F_i\sim_p\P^{h_i}(m_i)$, where
$m_i\le m$, $h_i\ge 0$, and $\sum_{i=1}^s(h_i+1)=h+1$.
\end{theorem}

\begin{remark}\label{remark_Adams}
 We will be especially interested in complexes $K\sim_p\P^2(m)$. Recall that Adams' Hopf invariant one theorem (see~\cite{Ada60}) implies that if $p=2$, then $m\in\{1,2,4,8\}$.
 We will conveniently write~$\RP^2$, $\CP^2$, $\HP^2$, and~$\OP^2$ instead of~$\P^2(1)$, $\P^2(2)$, $\P^2(4)$, and~$\P^2(8)$, respectively. On the contrary, if $p$ is odd, then a space $K\sim_p\P^2(m)$ exists for any even~$m$, see~\cite[Section~VII.4]{Bre72}.
\end{remark}

\begin{remark}
 We have formulated Theorems~\ref{thm_Smith} and~\ref{thm_Bredon} for simplicial actions on finite simplicial complexes, since this is the only case of interest for us. However, in all the papers and books cited above, these theorems are formulated in much greater generality, namely, for continuous actions on locally compact and paracompact Hausdorff topological spaces of finite cohomology dimension. In general, care must be taken with regard to the choice of the definition of a homology (or cohomology) manifold and the cohomology theory used. Namely, Alexander--Spanier, \v{C}ech, or sheaf theoretic cohomology are used in various versions of Theorems~\ref{thm_Smith} and~\ref{thm_Bredon}. Nevertheless, for finite simplicial complexes (and their open subsets) all these cohomology theories coincide with singular cohomology. A general definition of $\F_p$-cohomology manifold suitable for Theorem~\ref{thm_Smith} can be found in~\cite[Chapter~I]{Bor60}. It is based on the concept of local Betti numbers around a point, which goes back to Alexandroff~\cite{Ale35} and \v{C}ech~\cite{Cec34}. The fact that for finite simplicial complexes the general definition of $\F_p$-cohomology manifold is equivalent to Definition~\ref{defin_hm} follows from two simple observations. First, for a finite simplicial complex~$K$, the `clever' definition of local Betti numbers around a point~$x$ used in~\cite{Bor60} is equivalent to the `naive' definition
 $$
 \beta_i^{\F_p}(x)=\dim_{\F_p}H_i(|K|,|K|\setminus x;\F_p).
 $$
 Second, if $x$ lies in the relative interior of a $k$-simplex $\sigma\in K$, then
 $$
 H_i(|K|,|K|\setminus x;\F_p)\cong
 \left\{
 \begin{aligned}
  &\widetilde{H}_{i-k-1}\bigl(\link(\sigma,K);\F_p\bigr)&&\text{if }\link(\sigma,K)\ne\varnothing,\\
  &H_i(\sigma,\partial\sigma;\F_p)&&\text{if }\link(\sigma,K)=\varnothing,
 \end{aligned}
 \right.
 $$
 see~\cite[Lemma~63.1]{Mun84}.
\end{remark}

\section{Preliminaries on finite groups}\label{section_finite_groups}

\subsection{Sylow theorems}

Suppose that~$G$ is a finite group and $p$ is a prime.
Recall that if $p$ divides~$|G|$, then a \textit{Sylow $p$-subgroup} is a maximal $p$-subgroup of~$G$. We will need the following classical Sylow theorems:
\begin{itemize}
\item if $|G|=p^nm$, where $n>0$ and $p$ does not divide~$m$, then every Sylow $p$-subgroup of~$G$ has order~$p^n$,
\item all Sylow $p$-subgroups of~$G$ are conjugate to each other.
\end{itemize}

\subsection{Small groups}

In this paper, we will need a number of classification results on finite groups. We start with the following standard assertion.

\begin{propos}\label{propos_2p}
 Any group of order~$2p$, where $p$ is an odd prime, is isomorphic either to~$\rC_{2p}$ or to~$\rD_p$.
\end{propos}

Further, we need some specific classification results on groups of small order. For the convenience of the reader, we have collected all these results in  Proposition~\ref{propos_small_groups} below. Note that we are only interested in groups of order up to~$351$, so we definitely do not encounter any hard results on the classification of finite simple groups. The classification of finite groups of order up to~$2000$ was obtained by Besche, Eick, and O'Brien~\cite{BEOB01}. The databases of small groups are accessible in \textsf{GAP} (see~\cite{BEOB22}) or \textsc{Magma} (see~\cite[Chapter~62]{CBFS06}), see also a database~\cite{Dok}. All assertions of the following proposition can be immediately extracted from any of these databases. On the other hand, each of these assertions can be easily proved by hand using standard group-theoretic methods.

\begin{propos}\label{propos_small_groups}
Up to isomorphism,
 \begin{enumerate}
  \item  there are exactly $5$ groups of order~$8$, namely, $\rC_8$, $\rC_2\times\rC_4$, $\rC_2^3$, $\rD_4$, and the quaternion group~$\rQ_8$,
  \item  $\rC_2^4$ is the only group of order~$16$ that does not contain a subgroup isomorphic to either~$\rC_8$ or~$\rC_2\times\rC_4$,
  \item  there are exactly $5$ groups of order~$27$, namely, $\rC_{27}$, $\rC_3\times\rC_9$, $\rC_3^3$, and the two extraspecial groups~$\He_3=3^{1+2}_+$ and~$3^{1+2}_-$,
 \item  $\rS_3$, $\rA_4$, and~$\rS_4$ are the only groups of orders~$6$, $12$, and~$24$, respectively, that contain no elements of order~$6$,
 \item  $\rD_9$ and $\rC_3\rtimes \rS_3$ are the only two groups of orders~$18$ that contain no elements of order~$6$,
 \item  $\rC_3^2\rtimes \rC_4$ (with $\rC_4$ acting faithfully on~$\rC_3^2$) is the only group of order~$36$ that contains no elements of order~$6$,
 \item  $\PSU(3,2)\cong\rC_3^2\rtimes \rQ_8$ (with $\rQ_8$ acting faithfully on~$\rC_3^2$) is the only group of order~$72$ that contains no elements of orders~$6$ or~$8$,
 \item  $\rC_{13}\rtimes \rC_4$ (with $\rC_4$ acting faithfully on~$\rC_{13}$) is the only group of order~$52$ that contains no elements of order~$26$,
 \item  $\rC_3^3\rtimes\rC_{13}$ is the only group of order~$351$ that contains no elements of order~$39$.
\end{enumerate}
\end{propos}

\begin{remark}\label{remark_PSU}
Recall the definition of the group~$\PSU(n,p)$, where $p$ is a prime. (One can similarly define the groups~$\mathrm{PSU}(n,p^r)$ with $r>1$, but we do not need them.)  Let $\F_{p^2}$ be a field of $p^2$ elements and $F\colon x\mapsto x^p$ be the Frobenius automorphism  of it. Then the \textit{unitary group} $\mathrm{U}(n,p)$ is the subgroup of~$\mathrm{GL}(n,\F_{p^2})$ consisting of all matrices~$U$ that satisfy $F(U)^t\,U=I$, where $I$ is the identity matrix. The \textit{special unitary group} is the group
 $$
 \mathrm{SU}(n,p)=\mathrm{U}(n,p)\cap\mathrm{SL}(n,\F_{p^2}).
 $$
 Finally, the \textit{projective special unitary group}~$\PSU(n,p)$ is the quotient of the group~$\mathrm{SU}(n,p)$ by its centre, which consists of diagonal matrices.
In this paper we will deal only with the group~$\PSU(3,2)$. It can be checked directly that it has order~$72$ and is isomorphic to the semi-direct product $\rC_3^2\rtimes \rQ_8$ with $\rQ_8$ acting faithfully on~$\rC_3^2$. So a Sylow $2$-subgroup of~$\PSU(3,2)$ is isomorphic to~$\rQ_8$. Moreover, it is easy to see that $\PSU(3,2)$ has $9$ Sylow $2$-subgroups.
\end{remark}

\subsection{Finite groups in which every element has prime power order}

We will need two classical theorems on finite groups in which every element $\ne 1$ has prime power order. The first one is due to Suzuki (see~\cite[Theorem~16]{Suz62}); it classifies simple groups with this property. The second one is due to Higman (see~\cite[Theorem~1]{Hig57}); it describes the structure of solvable groups with this property.

\begin{theorem}[Suzuki]\label{thm_Suzuki}
 Up to isomorphism, there are exactly $8$ finite non-cyclic simple groups all of whose elements $\ne 1$ have prime power order, namely, the groups $\PSL(2,\F_q)$ for $q=5$, $7$, $8$, $9$, and~$17$, $\PSL(3,\F_4)$, and the Suzuki groups of Lie type~$\Sz(8)$ and~$\Sz(32)$.
\end{theorem}

\begin{remark}
 The proof of this theorem does not rely on the classification of finite simple groups.
\end{remark}

The groups $\PSL(2,\F_q)$ for $q=5$, $7$, $8$, $9$, and~$17$ have orders~$60$, $168$, $504$, $360$, and $2448$, respectively. The order of the group~$\PSL(3,\F_4)$ is $20160$. The Suzuki groups~$\Sz(q)$ have orders $q^2(q^2+1)(q-1)$. So the order of each group in the list from Theorem~\ref{thm_Suzuki} is divisible by at least one of the numbers~$5$, $7$, and~$17$.

\begin{cor}\label{cor_solvable}
 Let $G$ be a finite group all of whose elements $\ne 1$ have prime power order. Suppose that $|G|$ is not divisible by any of the numbers~$5$, $7$, and~$17$. Then $G$ is solvable.
\end{cor}

\begin{proof}
 All composition factors of~$G$ are simple finite groups in which every element $\ne 1$ has prime power order. It follows from Theorem~\ref{thm_Suzuki} that all of them are cyclic.
\end{proof}

%
%

\begin{theorem}[Higman]\label{thm_Higman}
Let $G$ be a finite solvable group all of whose elements $\ne 1$ have prime power order. Then   either $G$ is a $p$-group or $G$ satisfies the following:
\begin{enumerate}
 \item $|G|=p^mq^n$, where $p$ and~$q$ are distinct primes and $m,n\ge 1$,
 \item $G$ contains a non-trivial normal $p$-subgroup but does not contain a non-trivial normal $q$-subgroup,
 \item let $P$ be the  greatest normal $p$-subgroup of~$G$ and $p^k$ the order of~$P$; then $q^n$ divides $p^k-1$ and $p^{m-k}$ divides $q-1$.
\end{enumerate}
\end{theorem}

\begin{remark}
 All assertions of Theorem~\ref{thm_Higman}, except for the divisibility $ q^n\mid p^k-1$, are contained in Theorem~1 in~\cite{Hig57}, while this divisibility is implicitly contained in the proof of that theorem. For the convenience of the reader, let us prove that $ q^n\mid p^k-1$. Let $Q$ be a Sylow $q$-subgroup of~$G$. Then $Q$ acts by conjugation on~$P$. The restriction of this action to the set~$P\setminus\{1\}$ is free. Indeed, if $yxy^{-1}=x$ for some $x\in P\setminus\{1\}$ and $y\in Q\setminus\{1\}$, then the order of~$xy$ would be the product of the orders of~$x$ and~$y$ and hence not a prime power. Thus, $|Q|=q^n$ divides $|P\setminus\{1\}|=p^k-1$.
\end{remark}

\begin{remark}
Theorem~1 in~\cite{Hig57} also contains further results on the structure of~$G/P$, which will not be important for us.
\end{remark}

\section{Scheme of proof of Theorem~\ref{thm_main_hom}}\label{section_scheme}

The proof of Theorem~\ref{thm_main_hom} will consist of three steps. First, using the results by Smith, Bredon, Datta, and Bagchi (Theorems~\ref{thm_Smith},~\ref{thm_Bredon}, and~\ref{thm_Datta}), we will prove the following.

\begin{propos}\label{propos_prop}
 Suppose that $K$ is a $27$-vertex $\Z$-homology $16$-manifold such that the graded group~$H_*(K;\Z)$ is not isomorphic to~$H_*(S^{16};\Z)$. Let $V$ be the vertex set of~$K$. Then the symmetry group~$G=\Sym(K)$ has the following properties:
 \renewcommand{\labelenumi}{(\alph{enumi})}
 \begin{enumerate}
  \item The order of every element of~$G$ belongs to $\{1,2,3,4,9,11,13\}$.
  \item $G$ does not contain a subgroup isomorphic to either $\rC_2\times\rC_4$ or~$\rC_3\times\rC_9$.
  \item If $H\subset G$ is a $2$-subgroup, then $|H|\le 8$, and $V$ consists of exactly $15$, $9$, or~$6$ \ \mbox{$H$-orbits} in the cases $|H|=2$, $4$, or~$8$, respectively.
  \item Every $3$-subgroup $H\subset G$ acts freely on~$V$.
  \item Every element $g\in G$ of order~$11$ acts on~$V$ with two orbits of length~$11$ and five fixed points.
  \item If $G$ contains a subgroup~$H\cong\rD_{11}$, then $H$ acts on~$V$ with two orbits of length~$11$, two orbits of length~$2$, and one fixed point.
  \item Every element $g\in G$ of order~$13$ acts on~$V$ with two orbits of length~$13$ and one fixed point.
 \end{enumerate}
\end{propos}

The second part of the proof of Theorem~\ref{thm_main_hom} is purely group theoretic. Namely, we will classify all subgroups of~$\rS_{27}$ that satisfy the properties from Proposition~\ref{propos_prop}.

\begin{propos}\label{propos_group}
 Up to conjugation, there are exactly $27$ subgroups~$G\subset \rS_{27}$ that satisfy properties~(a)--(g) from Proposition~\ref{propos_prop}; they are the $26$ subgroups from Table~\ref{table_main} and besides the group~$\PSU(3,2)$ that acts on $V=\{1,\ldots,27\}$ with three orbits of length~$9$.
\end{propos}

Finally, we need to exclude the group~$\PSU(3,2)$.

\begin{propos}\label{propos_no_PSU}
 The group~$\PSU(3,2)$ cannot serve as the symmetry group of a $27$-vertex $\Z$-homology $16$-manifold~$K$ with the graded homology group~$H_*(K;\Z)$ not isomorphic to that of the $16$-sphere.
\end{propos}

Theorem~\ref{thm_main_hom} follows immediately from Propositions~\ref{propos_prop},~\ref{propos_group}, and~\ref{propos_no_PSU}. We will prove
these three propositions in Sections~\ref{section_prop}, \ref{section_group}, and~\ref{section_no_PSU}, respectively.

\section{Absence of elements of order~4 in the symmetry groups of 15-vertex homology 8-manifolds}\label{section_no_4}

The aim of this section is to obtain the following result, which will be used in the proofs of properties~(a) and~(b) from Proposition~\ref{propos_prop}.

\begin{propos}\label{propos_no_4}
 Suppose that $K$ is a $15$-vertex $\F_2$-homology $8$-manifold that satisfies complimentarity. Then the symmetry group~$\Sym(K)$ contains no elements of order~$4$.
\end{propos}

We start with the following proposition.

\begin{propos}\label{propos_490}
 Suppose that $K$ is a $15$-vertex $\F_2$-homology $8$-manifold that satisfies complimentarity. Then $K$ has exactly $490$ top-dimensional simplices.
\end{propos}

\begin{proof}
 Brehm and K\"uhnel~\cite[Section~1]{BrKu92} proved this proposition in the case of a combinatorial $8$-manifold. Let us show that their proof can be extended literally to the case of a homology manifold.

 Let $f_i$ denote the number of $i$-simplices of~$K$. We conveniently put $f_{-1}=1$. By Proposition~\ref{propos_comp_neigh} the simplicial complex~$K$ is $5$-neighborly. Hence
 \begin{equation}\label{eq_neigh}
  f_i=\binom{15}{i+1},\qquad -1\le i\le 4.
 \end{equation}
 By a theorem of Klee~\cite{Kle64} the numbers~$f_i$ satisfy the \textit{Dehn--Sommerville equations}
 \begin{equation}\label{eq_DS}
  h_{9-i}-h_i=(-1)^i\binom{9}{i}\bigl(\chi(K)-2\bigr),\qquad 0\le i\le 9,
 \end{equation}
 where $(h_0,\ldots,h_9)$ is the $h$-vector of~$K$ defined by
 \begin{equation}\label{eq_h}
  \sum_{i=0}^9h_it^{9-i}=\sum_{i=0}^9f_{i-1}(t-1)^{9-i}
 \end{equation}
 and $\chi(K)$ is the Euler characteristic of~$K$, that is,
 \begin{equation}\label{eq_Eul}
  \chi(K)=\sum_{i=0}^8(-1)^if_i.
 \end{equation}
 Brehm and K\"uhnel~\cite[Section~1]{BrKu92} showed that the system of equations~\eqref{eq_neigh}--\eqref{eq_Eul} has a unique solution and, for this solution, $f_8=490$.
\end{proof}

Proposition~\ref{propos_no_4} is the only assertion in the present paper whose proof is computer assisted. In~\cite{Gai22} the author suggested an algorithm that, being given numbers $d$, $n$, and~$N$ and a subgroup $G\subset S_n$, produces the list of all $G$-invariant weak $d$-pseudomanifolds~$K$ with $n$ vertices and at least $N$ maximal simplices such that $K$ satisfies `half of the complementarity condition': for each subset $\sigma\subset V$ at most one of the two subsets $\sigma$ and~$V\setminus\sigma$ is a simplex of~$K$. (Here $V$ is the set of vertices of~$K$.) This algorithm was implemented as a C++ program, see~\cite{Gai-prog}. It completes in a reasonable amount of time, provided that the symmetry group~$G$ is large enough. We could try to obtain Proposition~\ref{propos_no_4} by running this program for
$$
d=8,\qquad n=15,\qquad N=490,
$$
and subgroups~$G$ of~$S_{15}$ that are isomorphic to~$\rC_4$. (We need to try to take for~$G$ a representative of every conjugacy class of such subgroups.) If the result were the empty list, Proposition~\ref{propos_no_4} would follow. However, when run on such data, the program did not finish in a reasonable amount of time. Perhaps such calculations can still be performed using a high-performance computer, but we prefer to simplify the computational task by proving the following proposition.

\begin{propos}\label{propos_12simp}
 Suppose that $K$ is a $15$-vertex $\F_2$-homology $8$-manifold that satisfies complimentarity. Assume that the symmetry group~$\Sym(K)$ contains a subgroup~$H\cong\rC_4$. Then, renumbering the vertices of~$K$, we may achieve that the group~$H$ is generated by the permutation
 $$
 A=(1\ 3\ 2\ 4)(5\ 7\ 6\ 8)(9\ 11\ 10\ 12)
 $$
 and $K$ contains the following $12$ eight-dimensional simplices:
 \begin{align*}
  &\{1,2,3,4,5,6,7,8,15\},&
  &\{1,2,3,4,5,6,7,8,13\},\\
  &\{5,6,7,8,9,10,11,12,13\},&
  &\{5,6,7,8,9,10,11,12,14\},\\
  &\{1,2,3,4,9,10,11,12,14\},&
  &\{1,2,3,4,9,10,11,12,15\},\\
  &\{1,2,3,4,5,6,9,10,14\},&
  &\{1,2,3,4,5,6,9,10,15\},\\
  &\{1,2,5,6,7,8,9,10,15\},&
  &\{1,2,5,6,7,8,9,10,13\},\\
  &\{1,2,5,6,9,10,11,12,13\},&
  &\{1,2,5,6,9,10,11,12,14\}.
 \end{align*}
\end{propos}

\begin{remark}
 Certainly, once we achieve that $K$ contains the $12$ simplices listed in Proposition~\ref{propos_12simp}, we automatically obtain that $K$ also contains all simplices obtained from those $12$ simplices by the action of the group~$H$.
\end{remark}

Once we prove Proposition~\ref{propos_12simp}, we will be able to simplify our computational task considerably. First, we will need to check only one subgroup of~$\rS_{15}$ isomorphic to~$\rC_4$, namely, the subgroup generated by the permutation~$A$. Second, which is more important, we will be able to modify the algorithm by including the $12$ eight-dimensional simplices from Proposition~\ref{propos_12simp} (and all simplices in their $H$-orbits) in the simplicial complex~$K$ from the very beginning. This significantly reduces the running time, since including a simplex in~$K$ immediately prohibits including a lot of other simplices in~$K$, see~\cite[Remark~4.3]{Gai22}. Being implemented in this way, the program completes in a reasonable amount of time (about $3$ hours on one processor core of clock frequency 3.3~GHz), and yields the following result.

\begin{propos}\label{propos_no_pseudo}
 There are no $15$-vertex weak $8$-pseudomanifolds~$K$ that satisfy the following conditions:
 \begin{itemize}
 \item $K$ is invariant under the permutation~$A$ from Proposition~\ref{propos_12simp},
 \item for each subset $\sigma\subset V$, at least one of the subsets~$\sigma$ and~$V\setminus\sigma$ is not a simplex of~$K$ (where $V$ is the set of vertices of~$K$),
 \item $K$ has at least $490$ top-dimensional simplices,
 \item $K$ contains the $12$ top-dimensional simplices listed in Proposition~\ref{propos_12simp}.
 \end{itemize}
\end{propos}

Proposition~\ref{propos_no_4} follows immediately from Propositions~\ref{propos_12simp} and~\ref{propos_no_pseudo}. In the rest of this section we will prove Proposition~\ref{propos_12simp}. We start with the following auxiliary assertion.

\begin{propos}\label{propos_HP_2}
 Suppose that $K$ is a $15$-vertex $\F_2$-homology $8$-manifold that satisfies complimentarity and $g\in\Sym(K)$ is an element of order~$2$. Then $K^{\langle g\rangle}\cong\CP^2_9$.
\end{propos}

\begin{proof}
 Since $K$ satisfies complimentarity, Proposition~\ref{propos_comp_neigh} implies that $K$ is $5$-neighborly. Hence the union of any two $\langle g\rangle$-orbits is a simplex of~$K$. It follows that $K^{\langle g\rangle}$ is non-empty and connected. Then by Theorem~\ref{thm_Smith} the complex~$K^{\langle g\rangle}$ is an $\F_2$-homology manifold and hence a pseudomanifold. Therefore, $K^{\langle g\rangle}$ is not a simplex of positive dimension. Since $g$ acts on the vertices of~$K$ with at least $8$ orbits, Proposition~\ref{propos_complement} implies that $K^{\langle g\rangle}$ satisfies complementarity and has at least $8$ vertices. On the other hand, since $g\ne 1$, the complex~$K^{\langle g\rangle}$ has at most $14$ vertices. By Theorem~\ref{thm_Datta} we see that either $K^{\langle g\rangle}\cong\CP^2_9$ or $K^{\langle g\rangle}$ is a $14$-vertex $7$-dimensional pseudo-manifold that satisfies complementarity. The latter is impossible by Corollary~\ref{cor_no_odd_hom_mfld}, since $K^{\langle g\rangle}$ is an $\F_2$-homology manifold.
\end{proof}

\begin{proof}[Proof of Proposition~\ref{propos_12simp}]
Let $V=\{1,\ldots,15\}$ be the vertex set of~$K$. By Proposition~\ref{propos_comp_neigh} the simplicial complex~$K$ is $5$-neighborly.

Let $N\subset H$ be the subgroup isomorphic to~$\rC_2$. By Proposition~\ref{propos_HP_2} we have that $K^N\cong\CP^2_9$. Hence, $N$ acts on~$V$ with $9$ orbits, that is, with $6$ orbits of length~$2$ and $3$ fixed points.

We identify~$K^N$ with~$\CP^2_9$ along some isomorphism and identify the vertex set of~$\CP^2_9$ with the affine plane~$\mathcal{P}$ over~$\F_3$ as in Subsection~\ref{subsection_K}. These identifications provide a bijection $\varphi\colon V/N\to \mathcal{P}$. Then some $6$ points in~$\mathcal{P}$ correspond to $N$-orbits of length~$2$, and the other $3$ points in~$\mathcal{P}$ correspond to the $N$-fixed vertices of~$K$. We denote by $m\subset  \mathcal{P}$ the $3$-element subset consisting of the points that correspond to the $N$-fixed vertices of~$K$. Every $4$-element subset~$\sigma\subset \mathcal{P}$ that contains~$m$ corresponds to a $5$-element subset of~$V$, which is a simplex of~$K$, since $K$ is $5$-neighborly. Hence $\sigma$ is a simplex of~$\CP^2_9$. From Proposition~\ref{propos_3subset} it follows that $m$ is a non-special line.

Now, we can choose affine coordinates~$(x,y)$ in~$\mathcal{P}$ so that:
\begin{itemize}
 \item $m$ coincides with the line $\{x=0\}$,
 \item the special lines in~$\mathcal{P}$ are $\{y=0\}$, $\{y=1\}$, and $\{y=2\}$ in this cyclic order.
\end{itemize}

The action of~$H$ on~$K$ induces an action of the quotient group~$H/N\cong\rC_2$ on~$K^N=\CP^2_9$. We denote by~$s$ the automorphism of~$\CP^2_9$ given by the generator of~$H/N$; then $s^2=\mathrm{id}$. The group $\Sym(\CP^2_9)$ consists of affine transformations of~$\mathcal{P}$ that take special lines to special lines and preserve the cyclic order of them, see Subsection~\ref{subsection_K}. In particular, $s$ is such an affine transformation. Moreover, since the set of $N$-fixed points in~$V$ is $H$-invariant, we see that the line~$m$ is invariant under~$s$. Finally, since the action of~$H$ on~$V$ is faithfull, the set~$V$ contains an $H$-orbit of length~$4$; hence $s\ne\mathrm{id}$.
It follows easily that
$$
s(x,y)=(-x,y).
$$
Renumbering the vertices of~$K$, we can achieve that the bijection $\varphi\colon V/N\to\mathcal{P}$ is given by Table~\ref{table_bijection}. Then the group~$N$ is generated by the permutation
$$
B=(1\ 2)(3\ 4)(5\ 6)(7\ 8)(9\ 10)(11\ 12).
$$
Since a generator~$A$ of~$H$ satisfies $A^2=B$ and induces the involution~$s$ on~$\mathcal{P}$, we see that
$$
A=(1\ 3\ 2\ 4)^{\pm 1}(5\ 7\ 6\ 8)^{\pm 1}(9\ 11\ 10\ 12)^{\pm 1}
$$
for some signs~$\pm$. Swapping the vertices in the pairs~$\{3,4\}$, $\{7,8\}$, and~$\{11,12\}$, we can achieve that all signs are pluses.

\begin{table}
 \caption{The bijection between points $(x,y)\in\mathcal{P}$ and $N$-orbits in~$V$}\label{table_bijection}
 \begin{tabular}{|c|ccc|}
 \hline
 \diaghead(1,-1){abc}{$y$}{$x$} & $0$ & $1$ & $2$\\
 \hline
 $0$ & $\{13\}\vphantom{\Bigl(}$ & $\{1,2\}$  & $\{3,4\}$\\
 $1$ & $\{14\}\vphantom{\Bigl(}$ & $\{5,6\}$  & $\{7,8\}$\\
 $2$ & $\{15\}\vphantom{\Bigl(}$ & $\{9,10\}$ & $\{11,12\}$\\
 \hline
 \end{tabular}
\end{table}

By the construction of the simplicial complex~$\CP^2_9$, the following $12$ five-element subsets of~$\mathcal{P}$ are $4$-simplices of~$\CP^2_9$:
\begin{align*}
 &\bigl\{(0,2),(1,0),(2,0),(1,1),(2,1)\bigr\},&
 &\bigl\{(0,0),(1,0),(2,0),(1,1),(2,1)\bigr\},\\
 &\bigl\{(0,0),(1,1),(2,1),(1,2),(2,2)\bigr\},&
 &\bigl\{(0,1),(1,1),(2,1),(1,2),(2,2)\bigr\},\\
 &\bigl\{(0,1),(1,0),(2,0),(1,2),(2,2)\bigr\},&
 &\bigl\{(0,2),(1,0),(2,0),(1,2),(2,2)\bigr\},\\
 &\bigl\{(0,1),(1,0),(1,1),(1,2),(2,0)\bigr\},&
 &\bigl\{(0,2),(1,0),(1,1),(1,2),(2,0)\bigr\},\\
 &\bigl\{(0,2),(1,0),(1,1),(1,2),(2,1)\bigr\},&
 &\bigl\{(0,0),(1,0),(1,1),(1,2),(2,1)\bigr\},\\
 &\bigl\{(0,0),(1,0),(1,1),(1,2),(2,2)\bigr\},&
 &\bigl\{(0,1),(1,0),(1,1),(1,2),(2,2)\bigr\}.
\end{align*}
So the corresponding $12$ subsets of~$V/N$ are simplices of~$K^N$. Hence, the pre-images of these $12$ subsets under the projection $V\to V/N$ are simplices of~$K$. It is easy to check that they are exactly the $12$ simplices listed in Proposition~\ref{propos_12simp}.
\end{proof}

\section{Proof of Proposition~\ref{propos_prop}}\label{section_prop}

Throughout this section, $K$ is a $27$-vertex $\Z$-homology $16$-manifold such that the graded group~$H_*(K;\Z)$ is not isomorphic to~$H_*(S^{16};\Z)$, $V$ is the vertex set of~$K$, and $G=\Sym(K)$ is the symmetry group of~$K$. We identify~$V$ with the set~$\{1,\ldots,27\}$. Then $G$ is a subgroup of~$\rS_{27}$.

In this section we will study the group~$G$ and prove all properties from Proposition~\ref{propos_prop}. Namely, property~(a) is Proposition~\ref{propos_orders}, property~(b) follows from Propositions~\ref{propos_no_8_42} and~\ref{propos_no_27_93}, property~(c) is Proposition~\ref{propos_2} and Proposition~\ref{propos_no_16}, property~(d) is Corollary~\ref{cor_3^k}, property~(e) is Proposition~\ref{propos_11}, property~(f) is Corollary~\ref{cor_D11}, and property~(g) is Proposition~\ref{propos_13}.

From Theorems~\ref{thm_Novik} and~\ref{thm_ArMa} it follows that
\begin{itemize}
 \item $K\sim_p\OP^2$ for every prime~$p$,
 \item $K$ satisfies complementarity and hence is $9$-neighborly (by Proposition~\ref{propos_comp_neigh}).
\end{itemize}
We will make constant use of these two properties.

\subsection{Subgroups of order~$2^k$}\label{subsection_2k}

We start with the following standard lemma.

\begin{lem}\label{lem_split}
 Suppose that $H$ is a finite $p$-group that acts non-trivially on a finite set~$W$. Then there exists an index~$p$ normal subgroup $N\vartriangleleft H$ such that at least one $H$-orbit in~$W$ splits into $p$ different $N$-orbits.
\end{lem}

\begin{proof}
 It is well known that any index~$p$ subgroup of~$H$ is normal and any proper subgroup of~$H$ is contained in an index~$p$ subgroup. Let $\alpha\subseteq W$ be an $H$-orbit of length greater than~$1$ and $S\subset H$ be the stabilizer of an element $w\in \alpha$. Then $|S|< |H|$. We take for~$N$ an index~$p$ subgroup of~$H$ that contains~$S$. Then the orbit~$\alpha$ splits into $p$ different $N$-orbits.
\end{proof}

\begin{propos}\label{propos_2}
  Suppose that $H\subset G$ is a $2$-subgroup.
  \begin{enumerate}
   \item If $|H|=2$, then $H$ acts on~$V$ with $15$ orbits, $K^H$ is a $15$-vertex connected $\F_2$-homology $8$-manifold satisfying complementarity, and $K^H\sim_2\HP^2$.
   \item If $|H|=4$, then $H$ acts on~$V$ with $9$ orbits and $K^H\cong\CP^2_9$.
   \item If $|H|=8$, then $H$ acts on~$V$ with $6$ orbits and $K^H\cong\RP^2_6$.
  \end{enumerate}
\end{propos}

\begin{proof}
First, assume that $|H|\le 4$. Since $K$ is $9$-neighborly, the union of any two $H$-orbits is a simplex of~$K$. Therefore, $K^H$ is non-empty and connected, and besides, the number of vertices of~$K^H$ is equal to the number of $H$-orbits in~$V$. From Theorem~\ref{thm_Smith} it follows that $K^H$ is a connected $\F_2$-homology $d$-manifold for some $d$ and hence a $d$-pseudomanifold. By Proposition~\ref{propos_<d} we have $d<16$. By Theorem~\ref{thm_Bredon} and Remark~\ref{remark_Adams} we have $K^H\sim_2\mathbb{P}^2(m)$ for some $m\in\{1,2,4,8\}$. Since the fundamental cycle of~$K^H$ represents a non-trivial homology class in~$H_d(K^H;\F_2)$, we obtain that the homology dimension of~$K^H$ coincides with the geometric dimension, that is, $d=2m$. So $d\in\{2,4,8\}$. Since all $H$-orbits in~$V$ are simplices of~$K$, Proposition~\ref{propos_complement} implies that $K^H$ satisfies complementarity. By Theorem~\ref{thm_Datta} we obtain that $K^H$ is isomorphic to~$\RP^2_6$, or~$\CP^2_9$, or is a $8$-pseudomanifold with at least~$15$ vertices.

Assume that $|H|=2$. Then $H$ acts on~$V$ with at least $14$ orbits, so $K^H$ has at least $14$ vertices. Hence, $K^H$ is neither~$\RP^2_6$ nor~$\CP^2_9$. Therefore, $K^H$ is an $\F_2$-homology $8$-manifold with $n\ge 15$ vertices and $K^H\sim_2\HP^2$. Since $K^H$ satisfies complementarity, by Proposition~\ref{propos_comp_neigh} we obtain that $K^H$ is $(n-10)$-neighborly. If $n$ were greater than~$15$, we would obtain that $K^H$ is $6$-neighborly, so $H^4(K^H;\F_2)=0$, which would contradict the equivalence $K^H\sim_2\HP^2$. Thus, $n=15$.

Assume that $|H|=4$. Then $K^H$ has at least $7$ vertices, so $K^H$ is not isomorphic to~$\RP^2_6$. Hence, $K^H$ is either isomorphic to~$\CP^2_9$ or is a $8$-pseudomanifold with at least~$15$ vertices. By Lemma~\ref{lem_split} there is a subgroup~$N\vartriangleleft H$ such that $N\cong\rC_2$ and at least one $H$-orbit in~$V$ splits into two $N$-orbits. Then $K^H=(K^{N})^{H/N}$ and the group $H/N\cong\rC_2$ acts non-trivially on~$K^{N}$. By the above $K^{N}$ is a $15$-vertex connected  $\F_2$-homology $8$-manifold. Then, by Proposition~\ref{propos_<d}, $\dim K^H<\dim K^{N}=8$. Hence $K^H\cong\CP^2_9$. Therefore, $H$ acts on~$V$ with $9$ orbits.

Now, assume that $|H|=8$.  By Lemma~\ref{lem_split} there exists a subgroup $N\vartriangleleft H$ of order~$4$ such that at least one $H$-orbit in~$V$ splits into two $N$-orbits. Then the quotient group~$H/N$ acts non-trivially on~$K^H$. By the above we have $K^{N}\cong\CP^2_9$. So by Proposition~\ref{propos_BK} we have that
$$
K^H=(K^{N})^{H/N}\cong(\CP^2_9)^{\rC_2}\cong\RP^2_6.
$$
Since $K$ is $9$-neighborly, every $H$-orbit in~$V$ is a simplex of~$K$. Therefore, the number of $H$-orbits is equal to the number of vertices of~$K^H$, so~$H$ acts on~$V$ with $6$ orbits.
\end{proof}

\begin{cor}\label{cor_2}
 Every element $g\in G$ of order~$2$ fixes exactly three vertices of~$K$.
\end{cor}

\begin{propos}\label{propos_no_8_42}
 The group~$G$ does not contain a subgroup isomorphic to either~$\rC_8$ or~$\rC_2\times\rC_4$.
\end{propos}

\begin{proof}
 Assume that $H\subset G$ is a subgroup isomorphic to either~$\rC_8$ or~$\rC_2\times\rC_4$. Then $H$ contains a subgroup $N\cong\rC_2$ such that $H/N\cong\rC_4$. By Proposition~\ref{propos_2} we have that $K^N$ is a $15$-vertex $\F_2$-homology $8$-manifold satisfying complementarity.
 If the action of~$H/N$ on~$K^N$ were not faithful, then we would obtain a subgroup $Q\subset H$ of order~$4$ such that $Q\supset N$ and $K^Q=K^N$, which would yield a contradiction with Proposition~\ref{propos_2}. So the action of~$H/N$ on~$K^N$ is faithful, that is, the symmetry group $\Sym(K^N)$ contains the subgroup~$H/N\cong\rC_4$.  We get a contradiction with Proposition~\ref{propos_no_4}.
\end{proof}

Suppose that $H\subset G$ is a subgroup of order~$8$. From  Proposition~\ref{propos_no_8_42} and assertion~(1) of Proposition~\ref{propos_small_groups} it follows that $H$ is isomorphic to~$\rC_2^3$, or~$\rD_4$, or~$\rQ_8$. By Proposition~\ref{propos_2} the group~$H$ acts on~$V$ with $6$ orbits. The length of every orbit is a power of two and does not exceed~$8$. It is easy to see that there are exactly two ways to decompose~$27$ into the sum of $6$ powers of two that do not exceed~$8$, namely,
\begin{align*}
 27 &= 8 + 8 + 8 + 1+1+1\\
 {}&= 8+8+4+4+2+1.
\end{align*}
The following proposition says that the latter partition never occurs if $H$ is isomorphic to either~$\rC_2^3$ or~$\rQ_8$.

\begin{propos}\label{propos_8}
Suppose that $H$ is a subgroup of~$G$ that is isomorphic to either~$\rC_2^3$ or~$\rQ_8$. Then $H$ acts on~$V$ with $3$ orbits of length~$8$ and $3$ fixed points.
\end{propos}

\begin{proof}
 Assume the converse, i.\,e., $H$ acts on~$V$ with two orbits of length~$8$, two orbits of length~$4$, one orbit of length~$2$, and one fixed point. Let $\alpha$ be one of the two $H$-orbits of length~$4$. Consider the stabilizer $H_v$ of a vertex $v\in\alpha$. We have $H_v\cong\rC_2$. Since~$H$ is isomorphic to either~$\rC_2^3$ or~$\rQ_8$, it follows that  $H_v$ is a normal subgroup of~$H$. Hence $H_v$ fixes every vertex in~$\alpha$, and we obtain a contradiction with Corollary~\ref{cor_2}.
\end{proof}

\begin{propos}\label{propos_no_16}
 The group~$G$ does not contain a subgroup of order~$16$.
\end{propos}

\begin{proof}
Assume the converse, i.\,e., $G$ contains a subgroup~$H$ with $|H|=16$. From  Proposition~\ref{propos_no_8_42} and assertion~(2) of Proposition~\ref{propos_small_groups} it follows that $H\cong\rC_2^4$. By Lemma~\ref{lem_split} there exists a subgroup $N\subset H$ of order~$8$ such that at least one $H$-orbit in~$V$ splits into two $N$-orbits. Then
$$
K^H=(K^{N})^{H/N}\cong(\RP^2_6)^{\rC_2},
$$
where the action of~$\rC_2$ on~$\RP^2_6$ is nontrivial.
All subgroups isomorphic to~$\rC_2$ of the group $\Sym(\RP^2_6)\cong\rA_5$ are conjugate to each other. It is easy to check that $(\RP^2_6)^{\rC_2}\cong\pt\sqcup\partial\Delta^2$, cf.~\cite[Table~9.1]{Gai22}. Hence $K^H\cong\pt\sqcup\partial\Delta^2$. So from Proposition~\ref{propos_complement} it follows that all $H$-orbits are simplices of~$K$. Thus, the number of $H$-orbits is equal to the number of vertices of~$K^H$, that is, $H$ acts on~$V$ with $4$ orbits.
The number of ones in the binary representation of~$27$ is four, so $27$ can be written as the sum of four powers of two in a unique way, namely,
$$
27=16+8+2+1.
$$
It follows that $V$ decomposes into four $H$-orbits $\alpha_1$, $\alpha_2$, $\alpha_3$, and~$\alpha_4$ of lengths~$16$, $8$, $2$, and~$1$, respectively. Let $S\subset H$ be the stabilizer of a vertex $v\in\alpha_3$. Then $S\cong\rC_2^3$. Take any subgroup $Q\subset H$ such that $Q\cong\rC_2^3$ and $Q\ne S$.
Then $\alpha_3$ is a $Q$-orbit of length~$2$, and we obtain a contradiction with Proposition~\ref{propos_8}.
\end{proof}

\subsection{Elements of odd prime orders}

\begin{propos}\label{propos_odd_fix}
 Suppose that $H\subset G$ is a subgroup that is isomorphic to~$\rC_p$, where $p$ is an odd prime. Then every $H$-orbit $\sigma\subset V$ is a simplex of~$K$ and one of the following assertions holds for the fixed point simplicial complex~$K^H$:
 \begin{enumerate}
  \item $K^H\cong\pt\sqcup\partial\Delta^k$, where $k\ge 1$,
  \item $K^H\cong\CP^2_9$,
  \item $K^H$ is an $n$-vertex connected orientable $\F_p$-homology $d$-manifold such that
  \begin{itemize}
  \item $(d,n)=(8,15)$, or~$(12,19)$, or~$(12,21)$,
  \item $K^H$ satisfies complementarity,
  \item $H^*\bigl(K^H,\F_p\bigr)\cong\F_p[a]/(a^3)$, where $\deg a=d/2$.
  \end{itemize}
 \end{enumerate}
\end{propos}

\begin{proof}
  The group~$H=\rC_p$ cannot act on a $27$-element set with either one or two orbits. So from Proposition~\ref{propos_complement} it follows that $K^H$ is non-empty and either is a simplex of positive dimension or satisfies complementarity. However, by Theorem~\ref{thm_Smith} any connected component of~$K^H$ is an $\F_p$-homology manifold. Hence, $K^H$ cannot be a simplex of positive dimension. Therefore, $K^H$ satisfies complementarity. Moreover, from Proposition~\ref{propos_complement} it follows that every $H$-orbit $\sigma\subset V$ is a simplex of~$K$.

  Let $F_1,\ldots,F_s$ be the connected components of~$K^H$. Each vertex~$w$ of~$K^H$ is the barycentre of certain simplex $\sigma_{w}\in K$ that is an $H$-orbit. For every~$i=1,\ldots,s$, let $V_i$ be the union of all $H$-orbits~$\sigma_w$, where $w$ runs over the vertices of~$F_i$. Certainly, the subsets $V_1,\ldots,V_s$ are pairwise disjoint. Since all $H$-orbits in~$V$ are simplices of~$K$, we obtain that $V_1\cup\cdots\cup V_s=V$.  Moreover, $V_i\cup V_j$ is not a simplex of~$K$ unless $i=j$. Note that $s\le 3$. Indeed, otherwise both~$V_1\cup V_2$ and~$V_3\cup \cdots\cup V_s$ would be non-simplices of~$K$, which would contradict complementarity. Consider the three cases.
  \smallskip

  \textsl{Case 1: $s=3$.} Since $V_1\cup V_2$, $V_1\cup V_3$, and~$V_2\cup V_3$ are non-simplices of~$K$, the complementarity implies that $V_i\in K$ for $i=1,2,3$. So the connected components~$F_i$ are simplices. On the other hand, by Theorem~\ref{thm_Smith} the connected components~$F_i$ are $\F_p$-homology manifolds. It follows that in fact $F_1$, $F_2$, and~$F_3$ are points. So $K^H\cong \pt\sqcup\pt\sqcup\pt=\pt\sqcup\partial\Delta^1$.
  \smallskip

  \textsl{Case 2: $s=2$.} By complementarity exactly one of the two sets~$V_1$ and~$V_2$ is a simplex of~$K$. We may assume that $V_1\in K$ and~$V_2\notin K$. Then $F_1$ is a simplex. On the other hand, by Theorem~\ref{thm_Smith} we have that~$F_1$ is an $\F_p$-homology manifold. So $F_1=\pt$. Since $K^H$ satisfies complementarity, it follows that $F_2\cong\partial \Delta^k$ for some $k\ge 2$.
  \smallskip

  \textsl{Case 3: $s=1$.} By Theorem~\ref{thm_Smith} the simplicial complex~$K^H$ is a connected orientable $\F_p$-homology manifold of even dimension~$d=2m$.  Hence $K^H$ is an orientable pseudo-manifold of dimension~$2m$. By Theorem~\ref{thm_Bredon}, we have that
  \begin{equation}\label{eq_coh_iso}
  H^*\bigl(K^H;\F_p\bigr)\cong\F_p[a]/(a^3),\qquad \deg a =m.
  \end{equation}
  Since the product of odd-dimensional cohomology classes is anti-commutative and $p$ is odd, the isomorphism~\eqref{eq_coh_iso} can hold only if $m$ is even. Also $m>0$, since $K^H$ is connected. Therefore  $d$ is positive and divisible by~$4$. Moreover, by Proposition~\ref{propos_<d} we have $d<16$. Thus, $d$ is $4$, $8$,  or~$12$.

  If $d=4$, then by Theorem~\ref{thm_Datta} we have that $K^H\cong\CP^2_9$.

  Suppose that $d$ is either~$8$ or~$12$. Let $n$ be the number of vertices of~$K^H$ and $W$  the vertex set of~$K^H$. By Theorem~\ref{thm_Datta} we have that
  \begin{equation}\label{eq_ineq1}
   n\ge d+7.
  \end{equation}
 On the other hand, from~\eqref{eq_coh_iso} it follows that $K^H$ is not $m$-connected. Hence there exists an $(m+2)$-element subset $\eta\subset W$ such that $\eta\notin K^H$. By complementarity $W\setminus\eta\in K^H$. Therefore, $|W\setminus\eta|\le d+1$. Thus,
  \begin{equation}\label{eq_ineq2}
   n\le \frac{3d}2+3.
  \end{equation}
Combining inequalities~\eqref{eq_ineq1} and~\eqref{eq_ineq2}, we obtain that $n=15$ if $d=8$ and $19\le n\le 21$ if $d=12$.

Finally, from~\eqref{eq_coh_iso} it follows that $\chi\bigl(K^H\bigr)=3$. If $n$ were $20$, then by Proposition~\ref{propos_BD} we would obtain that $\chi\bigl(K^H\bigr)=1$. So $(d,n)\ne (12,20)$.
\end{proof}

\subsection{Absence of elements of prime orders $p\ne 2,3,11,13$}

Since $G$ is a subgroup of~$\rS_{27}$, it follows that $G$ does not contain elements of prime orders $p>27$.

\begin{propos}\label{propos_19_23}
 The group~$G$ contains no elements of orders~$19$ or~$23$.
\end{propos}
\begin{proof}
 Suppose that $p$ is either~$19$ or~$23$. Assume that there is an element $g\in G$ of order~$p$. Then the group $H=\langle g\rangle\cong \rC_p$ acts on the vertex set~$V$ with one orbit $\alpha$ of length~$p$ and $27-p$ fixed points. Since $\dim K=16$ and $p>17$, we have $\alpha\notin K$, which contradicts Proposition~\ref{propos_odd_fix}.
\end{proof}

\begin{propos}\label{propos_17}
 The group~$G$ contains no elements of order~$17$.
\end{propos}

\begin{proof}
 Assume that there is an element $g\in G$ of order~$17$. Then the group $H=\langle g\rangle\cong \rC_{17}$ acts on~$V$ with one orbit $\alpha$ of length~$17$ and $10$ fixed points. By Proposition~\ref{propos_odd_fix} we have $\alpha\in K$. Let $u_0$ be a vertex of~$\alpha$. Then $\beta=\alpha\setminus\{u_0\}$ is a $15$-simplex of~$K$. Since $K$ is a $16$-pseudomanifold, it follows that $\beta$ is contained in exactly two $16$-simplices of~$K$. One of these two simplices is~$\alpha$. The other has the form
 $$\beta\cup\{v\}=(\alpha\setminus\{u_0\})\cup\{v\},$$
 where $v\notin\alpha$. The group $H$ acts transitively on~$\alpha$ and fixes~$v$. Since $K$ is $H$-invariant, we see that
 $$(\alpha\setminus\{u\})\cup\{v\}\in K$$
 for all $u\in\alpha$.
 Thus, $K$ contains the boundary of the $17$-simplex $\alpha\cup\{v\}$, which is impossible, since $K$ is a $16$-pseudomanifold with $27$ vertices.
\end{proof}

\begin{propos}\label{propos_5_7}
 The group~$G$ contains no elements of orders~$5$ or~$7$.
\end{propos}

\begin{proof}
 Assume that there is an element $g\in G$ of order~$p$, where $p$ is either $5$ or~$7$. Let $H=\langle g\rangle$ and
 $$
 V=\alpha_1\sqcup\cdots\sqcup\alpha_s\sqcup\{v_{1},\ldots,v_t\},
 $$
 where $\alpha_1,\ldots,\alpha_s$ are the $H$-orbits with $|\alpha_i|=p$ and $v_{1},\ldots,v_t$ are the $H$-fixed points. We have $s\ge 1$, since $g\ne 1$, and $t\ge 1$, since $p$ is not a divisor of~$27$. Then
 \begin{equation}\label{eq_pst}
  ps+t=27.
 \end{equation}
 Since $K$ is $9$-neighborly, we see that $\alpha_i\cup\{v_j\}\in K$ and hence $\{b(\alpha_i),v_j\}\in K^H$ for all~$i$ and~$j$. So $K^H$ is a connected simplicial complex with $s+t$ vertices $$b(\alpha_1),\ldots,b(\alpha_s),v_1,\ldots,v_t.$$ By Proposition~\ref{propos_odd_fix}, either $K^H\cong\CP^2_9$ or $K^H$ is an $\F_p$-homology manifold satisfying assertion~(3) in Proposition~\ref{propos_odd_fix}. Consider the two cases.
 \smallskip

 \textsl{Case 1: $K^H\cong\CP^2_9$.} Then $s+t=9$. So it follows from equation~\eqref{eq_pst} that
 $$
 (p-1)s=18.
 $$
 Hence, $p=7$, $s=3$, and $t=6$. The combinatorial manifold~$\CP^2_9$ is $3$-neighborly, so $$\{b(\alpha_1),b(\alpha_2),b(\alpha_3)\}\in K^H$$ and hence $$\alpha_1\cup\alpha_2\cup\alpha_3\in K.$$
 We arrive at a contradicition, since $$\dim(\alpha_1\cup\alpha_2\cup\alpha_3)=20>16.$$
 \smallskip

\textsl{Case 2: $K^H$ is an $n$-vertex $\F_p$-homology $d$-manifold satisfying assertion~(3) from Proposition~\ref{propos_odd_fix}.} Then
\begin{equation}\label{eq_dn}
(d,n)=(8,15)\text{ or }(d,n)=(12, 19)\text{ or }(d,n)=(12,21).
\end{equation}
Combining the equality $s+t=n$ with~\eqref{eq_pst}, we obtain that
\begin{equation}\label{eq_stn}
 (p-1)s=27-n.
\end{equation}
Since $n\ge 15$ and $p\ge 5$, we have  $s\le 3$. Since $K^H$ satisfies complimentarity and $n\ge d+7$, it follows from Proposition~\ref{propos_comp_neigh} that $K^H$ is $5$-neighborly. Hence
$$
\sigma=\{b(\alpha_1),\ldots,b(\alpha_s),v_1,\ldots,v_{5-s}\}\in K^H.
$$
Since $K^H$ is a $d$-pseudomanifold, we obtain that $\sigma$ is contained in a $d$-simplex $\tau\in K^H$. We have
$$
\tau =\bigl\{b(\alpha_1),\ldots,b(\alpha_s),v_{i_1},\ldots,v_{i_{d+1-s}}\bigr\}
$$
for some $i_1<\cdots<i_{d+1-s}$. Then
$$
\alpha_1\cup\cdots\cup\alpha_s\cup \bigl\{v_{i_1},\ldots,v_{i_{d+1-s}}\bigr\}\in K.
$$
Since every simplex of~$K$ has at most~$17$ vertices, we obtain that
$$
(p-1)s+d\le 16.
$$
Combining this inequality with~\eqref{eq_stn}, we see that
$$
n-d\ge 11,
$$
which contradicts~\eqref{eq_dn}.
\end{proof}

\subsection{Elements of order $3$}

\begin{propos}\label{propos_3}
  Suppose that an element $g\in G$ has order~$3$. Then $g$ acts on~$V$ without fixed points and $K^{\langle g\rangle}\cong\CP^2_9$.
\end{propos}

\begin{proof}
Put $H=\langle g\rangle$. Since $K$ is $9$-neighborly, we see that the union of any two $H$-orbits is a simplex of~$K$. Therefore any two vertices are joined by an edge in~$K^H$. So $K^H$ is connected. If we prove that $g$ acts on~$V$ without fixed points, then it will follow that $K^H$ has $9$ vertices; hence Proposition~\ref{propos_odd_fix} will imply that $K^H\cong\CP^2_9$.

Let us prove that $g$ acts on~$V$ without fixed points. Assume the converse, i.\,e., that $V$ contains an $H$-fixed point and then at least three $H$-fixed points. The number of vertices of~$K^H$ is equal to the number of $H$-orbits in~$V$ and hence is greater than~$9$. From Proposition~\ref{propos_odd_fix} it follows that $K^H$ is an $n$-vertex $d$-pseudomanifold satisfying complementarity, where $(d,n)$ is $(8,15)$ or~$(12,19)$ or~$(12,21)$. Consider two cases.
\smallskip

\textsl{Case 1: $d=8$.} Then $n=15$. Hence $H$ acts on~$V$ with exactly $15$ orbits, so with  $6$ orbits $\alpha_1,\ldots,\alpha_6$ of length~$3$ and $9$ fixed points. Since $K^H$ satisfies complementarity, from Proposition~\ref{propos_comp_neigh} it follows that $K^H$ is $5$-neighborly. Therefore, the set $\sigma=\{b(\alpha_1),\ldots,b(\alpha_5)\}$ is a simplex of~$K^H$. But $K^H$ is an $8$-pseudomanifold, so $\sigma$ must be contained in an $8$-simplex. This means that the union of the five orbits $\alpha_1,\ldots,\alpha_5$ together with some other four orbits is a simplex of~$K$. However, such union has cardinality at least~$19$, which gives a contradiction, since $\dim K=16$.
\smallskip

\textsl{Case 2: $d=12$.} Then $n$ is either $19$ or~$21$. Hence $H$ acts on~$V$ either with  $4$ orbits $\alpha_1,\alpha_2,\alpha_3,\alpha_4$ of length~$3$ and $15$ fixed points or with $3$ orbits $\alpha_1,\alpha_2,\alpha_3$ of length~$3$ and $18$ fixed points. Since $K^H$ satisfies complementarity, from Proposition~\ref{propos_comp_neigh} it follows that $K^H$ is $5$-neighborly. Therefore, the set $\sigma=\{b(\alpha_1),b(\alpha_2),b(\alpha_3)\}$ is a simplex of~$K^H$. But $K^H$ is a $12$-pseudomanifold, so $\sigma$ must be contained in a $12$-simplex. This means that the union of the three orbits $\alpha_1,\alpha_2,\alpha_3$ together with some other ten $H$-orbits is a simplex of~$K$. However, such union has cardinality at least~$19$, which again gives a contradiction, since $\dim K=16$.
\end{proof}

\begin{cor}\label{cor_3^k}
Any $3$-subgroup~$H\subset G$ acts freely on~$V$.
\end{cor}

\begin{cor}\label{cor_no_81}
 $|G|$ is not divisible by~$81$.
\end{cor}

\begin{propos}\label{propos_no_27_93}
 The group~$G$ does not contain a subgroup isomorphic to either~$\rC_{27}$ or~$\rC_3\times\rC_9$.
\end{propos}

\begin{proof}
 Assume that $H\subset G$ is a subgroup isomorphic to either~$\rC_{27}$ or~$\rC_3\times\rC_9$. Then $H$ contains a subgroup~$N\cong\rC_3$ such that $H/N\cong \rC_9$. By Proposition~\ref{propos_3}, we have $K^N\cong\CP^2_9$. By Corollary~\ref{cor_3^k}, the group~$H$ acts on~$V$ freely and hence transitively. Therefore, the group $H/N\cong\rC_9$ acts on~$K^N$ simplicially and transitively on vertices. We arrive at a contradiction, since the Sylow $3$-subgroup of the group $\Sym(\CP^2_9)$ is isomorphic to the Heisenberg group~$\mathrm{He}_3=3^{1+2}_+$ (see~\cite{KuBa83}), which contains no elements of order~$9$.
\end{proof}

\subsection{Absence of elements of order~$6$}

\begin{propos}\label{propos_6}
 The group~$G$ contains no elements of order~$6$.
\end{propos}

\begin{proof}
 Assume that $G$ contains an element of order~$6$. Then~$G$ contains a subgroup
 $$H\cong\rC_6\cong\rC_2\times\rC_3.$$
 Using Propositions~\ref{propos_3}, we obtain that
 $$
 K^H=(K^{\rC_3})^{\rC_2}\cong (\CP^2_9)^{\rC_2}.
 $$
 If the action of~$\rC_2$ on~$\CP^2_9$ is trivial, then $K^H\cong\CP^2_9$, and if the action of~$\rC_2$ on~$\CP^2_9$ is non-trivial, then $K^H\cong\RP^2_6$ (by Proposition~\ref{propos_BK}). So $K^H$ has either~$9$ or~$6$ vertices.
 On the other hand, by Proposition~\ref{propos_2} the factor~$\rC_2\subset H$ acts on~$V$ with three fixed points and $12$ orbits of length~$2$. Further, the action of the factor~$\rC_3$ on~$V$ is free and commutes with the action of the factor~$\rC_2$. It follows that $\rC_3$ acts freely on the $15$-element set $V/\rC_2$. Hence $|V/H|=5$. Therefore, $K^H$ has at most $5$ vertices. The obtained contradiction shows that $G$ contains no elements of order~$6$.
 \end{proof}

\subsection{Elements of order~$11$}

\begin{propos}\label{propos_11}
  Suppose that an element $g\in G$ has order~$11$. Then $g$ acts on~$V$ with $5$ fixed points and two orbits of length~$11$ and $K^{\langle g\rangle}\cong\pt\sqcup\partial\Delta^5$, where the separate point~$\pt$ is the barycentre of one of the two orbits of length~$11$.
\end{propos}

\begin{proof}
 Let $H=\langle g\rangle\cong\rC_{11}$.
 First, we need to exclude the possibility that~$H$ acts on~$V$ with one orbit $\alpha$ of length~$11$ and $16$ fixed points. Assume that this is the case. By Proposition~\ref{propos_odd_fix} we have that $\alpha\in K$. Then the simplicial complex $K^H$ has $17$ vertices, namely, the barycentre~$b(\alpha)$ and the $16$ fixed vertices of~$K$. Since $K$ is $9$-neighborly, any two vertices in~$V\setminus\alpha$ are connected by an edge in~$K$ and hence in~$K^H$. Besides, since $K$ is a $16$-pseudomanifold, the $10$-simplex~$\alpha$ is contained in a $16$-simplex of~$K$. Hence, $b(\alpha)$ is contained in a $6$-simplex of~$K^H$. Consequently, $K^H$ is connected.  However, then from Proposition~\ref{propos_odd_fix} it follows that $K^H$ must have $9$, or~$15$, or~$19$, or~$21$ vertices, so we arrive at a contradiction.

 Thus, $H$ acts on~$V$ with two orbits $\alpha$ and~$\beta$ of length~$11$, and $5$ fixed points. By Proposition~\ref{propos_odd_fix} we have $\alpha\in K$ and~$\beta\in K$. Then the simplicial complex $K^H$ has $7$ vertices, namely, the barycentres~$b(\alpha)$ and~$b(\beta)$ and the $5$ fixed points~$v_1,\ldots,v_5$ in~$V$. From Proposition~\ref{propos_odd_fix} it follows that $K^H\cong\pt\sqcup\partial\Delta^5$. Any two vertices $v_i$ and~$v_j$ are connected by an edge in~$K$ and hence in~$K^H$. So the separate point~$\pt$  is either~$b(\alpha)$ or~$b(\beta)$.
 \end{proof}

 \begin{propos}\label{propos_22}
 Suppose that $H\subset G$ is a subgroup that is isomorphic to one of the groups~$\rC_{22}$, $\rD_{11}$, or~$\rC_{33}$. Then $V$ contains two $H$-orbits of length~$11$.
\end{propos}

\begin{proof}
  In each of the three cases, the group~$H$ contains a normal subgroup~$N\cong\rC_{11}$. The quotient group~$H/N$ is isomorphic to either~$\rC_2$ or~$\rC_3$. By Proposition~\ref{propos_11} the group~$N$ acts on~$V$ with two orbits~$\alpha$ and~$\beta$ of length~$11$ and five fixed points, and $K^N\cong\pt\sqcup\partial\Delta^5$, where the separate point~$\pt$ is the barycentre of either~$\alpha$ or~$\beta$. We may assume that $\pt=b(\alpha)$. The action of~$H$ on~$K$ induces an action of~$H/N$ on~$K^N$. Certainly, the separate point~$\pt=b(\alpha)$ is fixed under this action and hence the point~$b(\beta)$ is also fixed. Consequently, $\alpha$ and~$\beta$ are $H$-orbits.
\end{proof}

\begin{cor}\label{cor_22}
 The group~$G$ contains no elements of orders~$22$ or~$33$.
\end{cor}

\begin{proof}
 Assume that an element $g\in G$ has order~$22$ or~$33$ and let $H=\langle g\rangle$. Then the element~$g^{11}$ has order~$2$ or~$3$. By Proposition~\ref{propos_22} the set~$V$ contains two $H$-orbits~$\alpha$ and~$\beta$ of length~$11$. Since neither~$2$ nor~$3$ divides~$11$, the element~$g^{11}$ fixes at least one point in each of the $H$-orbits~$\alpha$ and~$\beta$. Since $H$ is commutative, it follows that $g^{11}$ fixes every point in~$\alpha$ and every point in~$\beta$, which contradicts either Corollary~\ref{cor_2} or Proposition~\ref{propos_3}.
\end{proof}

\begin{cor}\label{cor_D11}
 Suppose that $H\subset G$ is a subgroup isomorphic to~$\rD_{11}$. Then $H$ acts on~$V$ with two orbits of length~$11$, two orbits of length~$2$, and one fixed point.
\end{cor}

\begin{proof}
  By Proposition~\ref{propos_22} the set~$V$ contains two $H$-orbits~$\alpha$ and~$\beta$ of length~$11$. Let $v_1,\ldots,v_5$ be the five points of~$V$ not belonging to~$\alpha\cup\beta$.
  Let $g\in H$ be an element of order~$2$. Then $g$ fixes an odd number of points in~$\alpha$ and an odd number of points in~$\beta$. On the other hand, by Corollary~\ref{cor_2} the element~$g$ fixes exactly three points in~$V$. Therefore $g$ fixes exactly one of the points~$v_1,\ldots,v_5$. We may assume that $g$ fixes~$v_5$ and swaps the points in each of the pairs~$\{v_1,v_2\}$ and~$\{v_3,v_4\}$. Since the length of any $H$-orbit must be a divisor of~$22$, we conclude that $\{v_1,v_2\}$, $\{v_3,v_4\}$, and~$\{v_5\}$ are $H$-orbits.
\end{proof}

 \subsection{Elements of order~$13$}

\begin{propos}\label{propos_13}
  Suppose that an element $g\in G$ has order~$13$. Then $g$ acts on~$V$ with two orbits of length~$13$ and one fixed point, and $K^{\langle g\rangle}$ is the disjoint union of three points.
\end{propos}

\begin{proof}
Let $H=\langle g\rangle\cong\rC_{13}$. We need to exclude the possibility that~$H$ acts on~$V$ with one orbit $\alpha$ of length~$13$ and $14$ fixed points. Assume that this is the case.
By Proposition~\ref{propos_odd_fix} we have that $\alpha\in K$. Then the simplicial complex $K^H$ has $15$ vertices, namely, the barycentre~$b(\alpha)$ and the $14$ fixed points. Since $K$ is $9$-neighborly, any two vertices in~$V\setminus\alpha$ are connected by an edge in~$K$ and hence in~$K^H$. Besides, since $K$ is a $16$-pseudomanifold, the $12$-simplex~$\alpha$ is contained in a $16$-simplex of~$K$. Hence, $b(\alpha)$ is contained in a $4$-simplex of~$K^H$. Consequently, $K^H$ is connected. Then from Proposition~\ref{propos_odd_fix} it follows that $K^H$ is an  $\F_{13}$-homology $8$-manifold and hence an $8$-pseudomanifold. On the other hand, any $9$ vertices in~$V\setminus\alpha$ form a simplex of~$K$ and hence a simplex of~$K^H$. Hence, the complex~$K^H$ contains a subcomplex isomorphic to the $8$-skeleton of a $13$-simplex. Therefore, $K^H$ is not an $8$-pseudomanifold, which gives a contradiction.

 Thus, $H$ acts on~$V$ with two orbits of length~$13$ and one fixed point. From Proposition~\ref{propos_odd_fix} it follows that $K^H$ is the disjoint union of three points.
 \end{proof}

\begin{propos}\label{propos_26}
The group $G$ contains no elements of orders~$26$ or~$39$.
\end{propos}
\begin{proof}
Assume that an element $g\in G$ has order either~$26$ or~$39$, let $H=\langle g\rangle$, and let  $N\subset H$ be the subgroup isomorphic to~$\rC_{13}$. Then the order of the element~$g^{13}$ is either~$2$ or~$3$. By Proposition~\ref{propos_13} the group $N$ acts on~$V$ with two orbits~$\alpha$ and~$\beta$ of length~$13$ and one fixed point~$v$. Hence either both $\alpha$ and~$\beta$ are $H$-orbits or $\alpha\cup\beta$ is an $H$-orbit. In the former case, as in the proof of Corollary~\ref{cor_22} we obtain that $g^{13}$ fixes every point in~$\alpha$ and every point in~$\beta$. In the latter case, the element~$g^{13}$ swaps~$\alpha$ and~$\beta$ and hence fixes exactly one point~$v$ in~$V$. In both cases we obtain a contradiction with either Corollary~\ref{cor_2} or Proposition~\ref{propos_3}.
\end{proof}

\subsection{End of the proof of property~(a)}

\begin{propos}\label{propos_orders}
The order of every element of~$G$ belongs to $\{1,2,3,4,9,11,13\}$.
\end{propos}

\begin{proof}
From Propositions~\ref{propos_19_23}, \ref{propos_17}, and~\ref{propos_5_7} it follows that the only primes that could occur among the orders of elements of~$G$ are $2$, $3$, $11$, and~$13$. Let us  show that $G$ does not contain elements of orders~$pq$, where $p$ and~$q$ are different primes from~$\{2,3,11,13\}$. For all pairs~$\{p,q\}$, except for~$\{11,13\}$, this follows from Propositions~\ref{propos_6} and~\ref{propos_26} and Corollary~\ref{cor_22}. So we need to show that $G$ contains no elements of order~$143$. If $g$ were such an element, then the permutation~$g$ would be the product of two disjoint cycles of lengths~$11$ and~$13$, respectively. Hence $g^{13}$ would by a cycle of length~$11$, which is impossible by Proposition~\ref{propos_11}. Therefore, the order of every element of~$G$ is a prime power~$p^k$, where $p\in\{2,3,11,13\}$. Further, by Propositions~\ref{propos_no_8_42} and~\ref{propos_no_27_93}, the group~$G$ does not contain elements of orders~$8$ or~$27$. Also, $G$ contains no elements of orders~$121$ or~$169$, since there are no such elements in~$\rS_{27}$. The proposition follows.
\end{proof}

\section{Proof of Proposition~\ref{propos_group}}\label{section_group}
The aim of this section is to prove Proposition~\ref{propos_group}.
This is a purely group theoretic result, so we may completely forget about triangulations.
It is easy to check that all subgroups $G\subset\rS_{27}$ listed in Table~\ref{table_main} and also the group~$\PSU(3,2)=\rC_3^2\rtimes\rQ_8$ acting on~$\{1,\ldots,27\}$ with three orbits of length~$9$ satisfy properties~(a)--(g) from Proposition~\ref{propos_prop}. So we only need to prove that every subgroup~$G$ satisfying these properties is in the list.

\begin{remark}
 The first part of the proof below (up to and including Proposition~\ref{propos_order_list}) follows the approach suggested to the author by Andrey Vasil'ev. A key feature of this approach is the use of the solvability of the group~$G$. The author's original proof was based on counting elements of different orders using the Sylow theorems; it was more cumbersome.
\end{remark}

Throughout this section, we suppose that $G\subset \rS_{27}$ is a non-trivial subgroup that satisfies properties~(a)--(g) from Proposition~\ref{propos_prop}. We put $V=\{1,\ldots,27\}$.

\begin{propos}\label{propos_order_solv}
 The order of~$G$ divides $
 2^3\cdot 3^3\cdot 11\cdot 13
 $
and $G$ is solvable.
\end{propos}

\begin{proof}
  From property~(a) it follows that $2$, $3$, $11$, and~$13$ are the only primes that may enter the prime factorization of~$|G|$. By property~(c) the exponent of~$2$ in the prime factorization of~$|G|$ does not exceed~$3$. By property~(d) a Sylow $3$-subgroup of~$G$ acts freely on~$27$ points. Hence the exponents of~$3$ in the prime factorization of~$|G|$ also does not exceed~$3$. Now, let $p$ be either~$11$ or~$13$.  The exponent of~$p$ in $|\rS_{27}|=27!$ is equal to~$2$, so any subgroup $H\subset\rS_{27}$ of order~$p^2$ is a Sylow $p$-subgroup of~$\rS_{27}$. By the Sylow theorems all Sylow $p$-subgroups are conjugate to each other. Hence any such subgroup~$H$ is generated by two disjoint $p$-cycles. So if $|G|$ were divisible by~$p^2$, then $G$ would contain a $p$-cycle, that is, a permutation of order~$p$ that acts on~$V$ with one orbit of length~$p$ and~$27-p$ fixed points. Therefore, we would obtain a contradiction with properties~(e) or~(g). Consequently, the exponents of~$11$ and~$13$ in the prime factorization of~$|G|$ do not exceed~$1$. Thus, $|G|$ divides $2^3\cdot 3^3\cdot 11\cdot 13.$

  By property~(a) all elements of~$G$ have prime power order. Since  $|G|$ is not divisible by any of the primes~$5$, $7$, and~$17$, it follows from Corollary~\ref{cor_solvable}  that $G$ is solvable.
\end{proof}

\begin{remark}
 Instead of using Suzuki's Theorem~\ref{thm_Suzuki}, the solvability of~$G$ can be proven in another way. Indeed, the list of all non-cyclic finite simple groups of order $\le 2^3\cdot 3^3\cdot 11\cdot 13=30888$ is well known and consists of only $23$ groups. (Note that, in contrast to the general problem of classifying finite simple groups, listing finite simple groups of such a small order does not rely on any hard results, cf.~\cite{Hal72}.)  It is easy to check that for none of these simple groups its order divides $30888$. Therefore, any group whose order divides $30888$ is solvable.
\end{remark}

\begin{propos}\label{propos_order_list}
 The order of~$G$ is one of the numbers~$2$, $3$, $4$, $6$, $8$, $9$, $11$, $12$, $13$, $18$, $22$, $24$, $26$, $27$, $36$, $39$, $52$, $54$, $72$, or~$351$.
\end{propos}

\begin{proof}
 By property~(a) and Proposition~\ref{propos_order_solv},  $G$ is a solvable group all of whose elements $\ne 1$ have prime power order, and moreover, $|G|$ divides $2^3\cdot 3^3\cdot 11\cdot 13$. If $G$ is a $p$-group, then its order is either~$2^m$ or~$3^m$ with $1\le m\le 3$, or~$11$, or~$13$. Suppose that $G$ is not a $p$-group. Then from Theorem~\ref{thm_Higman} it follows that
 \begin{enumerate}
 \item $|G|=p^mq^n$, where $p$ and~$q$ are distinct primes from~$\{2,3,11,13\}$, $1\le m\le 3$, $1\le n\le 3$, and moreover $m=1$ whenever $p$ is either~$11$ or~$13$ and $n=1$ whenever $q$ is either~$11$ or~$13$,
 \item $G$ contains a non-trivial normal $p$-subgroup but does not contain a non-trivial normal $q$-subgroup,
 \item if $p^k$ is the order of the  greatest normal $p$-subgroup~$P$ of~$G$, then $q^n$ divides $p^k-1$ and $p^{m-k}$ divides $q-1$.
\end{enumerate}

Consider four cases.
\smallskip

\textsl{Case 1: $p=2$.} Then $2^k-1$ is $1$, $3$, or~$7$. Since $q^n$ divides $2^k-1$, we have that $k=2$, $q=3$, and $n=1$. Hence $|G|$ is either~$12$ or~$24$.
\smallskip

\textsl{Case 2: $p=3$.} Then $3^k-1$ is $2$, $8$, or~$26$. Since $q^n$ divides $3^k-1$, we have the following two subcases.
\smallskip

\textsl{Subcase 2a: $q=2$ and $(k,n)$ is one of the pairs $(1,1)$, $(2,1)$, $(2,2)$, $(2,3)$, or~$(3,1)$.} Then $m=k$, since $2^{m-k}$ must divide $q-1=1$. Hence $|G|$ is~$6$, $18$, $36$, $72$, or~$54$.
\smallskip

\textsl{Subcase 2b: $q=13$, $k=3$, and $n=1$.} Then $|G|=351$.
\smallskip

\textsl{Case 3: $p=11$.} Then $k=m=1$, so $q^n$ divides $p-1=10$. Hence $q=2$ and $n=1$. Therefore $|G|=22$.
\smallskip

\textsl{Case 4: $p=13$.} Then $k=m=1$, so $q^n$ divides $p-1=12$. Hence $q^n$ is $2$, $3$, or~$4$. Therefore $|G|$ is $26$, $39$, or~$52$.
\end{proof}

\begin{propos}\label{propos_2_4}
 Every element~$g\in G$ of order either~$2$ or~$4$ has exactly three fixed points in~$V$.
\end{propos}

\begin{proof}
First, assume that $g$ has order~$2$. Then by property~(c) the group~$\langle g\rangle$ acts on~$V$ with exactly $15$ orbits. Hence, $g$ fixes exactly three points.

Second, assume that $g$ has order~$4$. By property~(c) the group~$\langle g\rangle$ acts on~$V$ with exactly $9$ orbits. There are two possibilities:
\begin{itemize}
 \item $\langle g\rangle$ acts on~$V$ with six orbits of length~$4$ and three fixed points,
 \item $\langle g\rangle$ acts on~$V$ with five orbits of length~$4$, three orbits of length~$2$, and one fixed point.
\end{itemize}
However, in the latter case we would obtain that the element~$g^2$ fixes seven points in~$V$, so the latter case is impossible.
\end{proof}

\begin{proof}[Proof of Proposition~\ref{propos_group}]
 Let us consider in turn all the possibilities for~$|G|$ from Proposition~\ref{propos_order_list}.
 \smallskip

\textsl{Case 1: $|G|=2^k$, where $1\le k\le 3$.}
By properties~(a) and~(b) and assertion~(1) of Proposition~\ref{propos_small_groups} we obtain that $G$ is isomorphic to one of the groups~$\rC_2$,  $\rC_2^2$, $\rC_4$, $\rC_2^3$, $\rD_4$, or~$\rQ_8$.

If $G$ is isomorphic to either~$\rC_2$ or~$\rC_4$, then  from Proposition~\ref{propos_2_4} it follows that the group~$G$ fixes exactly $3$ points in~$V$ and acts freely on the other $24$ points. These two actions are present in Table~\ref{table_main}.

Suppose that $G\cong\rC_2^2$. By Proposition~\ref{propos_2_4} every non-trivial element $g$ of~$G$ fixes exactly three points in~$V$. It follows easily that up to isomorphism there are exactly two actions of~$\rC_2^2$ that satisfy this property:
\begin{itemize}
 \item the action with six orbits of length~$4$ and three fixed points,
 \item the action with five orbits of length~$4$, three orbits of length~$2$, and one fixed point such that the stabilizers of points in the three orbits of length~$2$ are pairwise different.
\end{itemize}
Both actions are present in Table~\ref{table_main}.

Finally, suppose that $|G|=8$. Then $G$ is isomorphic to~$\rC_2^3$, or~$\rD_4$, or~$\rQ_8$. By property~(c) the group~$G$ acts on~$V$ with exactly $6$ orbits. Moreover, every $G$-orbit has length~$1$, $2$, $4$, or~$8$. As we have already mentioned in Subsection~\ref{subsection_2k},
there are exactly two ways to decompose~$27$ into the sum of $6$ powers of two that do not exceed~$8$, namely,
\begin{align*}
 27 &= 8 + 8 + 8 + 1+1+1\\
 {}&= 8+8+4+4+2+1.
\end{align*}
The actions of the three possible groups~$G$ with three orbits of length~$8$ and three fixed points are  present in Table~\ref{table_main}.

Now, we consider the case of $G$ acting on~$V$ with two orbits of length~$8$, two orbits~$\alpha$ and~$\beta$ of length~$4$, one orbit~$\gamma$ of length~$2$, and one fixed point~$v_0$.
As in the proof of Propostion~\ref{propos_8}, we see that if $G$ were isomorphic to either~$\rC_2^3$ or~$\rQ_8$, then $G$ would contain a non-trivial element that fixes every point of~$\alpha$, which would contradict Proposition~\ref{propos_2_4}. Therefore  $G\cong\rD_4$. From Proposition~\ref{propos_2_4} it follows that each subgroup $\rC_2\subset\rD_4$ fixes exactly two points in~$\alpha\cup\beta\cup\gamma$. The group~$\rD_4$ has exactly $5$ subgroups isomorphic to~$\rC_2$, one of which is the centre~$Z$. If $Z$ fixed a point in an orbit of length~$4$, then it would fix every point in this orbit, so we would again arrive at a contradiction with Proposition~\ref{propos_2_4}. Hence, $Z$ fixes the two points of~$\gamma$. Then each of the $8$ points in~$\alpha\cup\beta$ must be fixed by some non-central subgroup $\rC_2\subset\rD_4$. Therefore, each of the four non-central subgroups $\rC_2\subset\rD_4$ fixes exactly two points in~$\alpha\cup\beta$. These two points lie in the same $G$-orbit (either~$\alpha$ or~$\beta$), since they must be swapped by the generator of~$Z$. Thus, the stabilizers~$G_u$ and~$G_v$ of any points $u\in\alpha$ and $v\in\beta$ are not conjugate to each other. Moreover, the stabilizer~$G_w$ of any point $w\in \gamma$ is a subgroup of order~$4$ that does not contain a non-central subgroup $\rC_2\subset\rD_4$. Therefore, $G_w=\rC_4$.
The obtained action is present in Table~\ref{table_main}.
\smallskip

\textsl{Case 2: $|G|=3^k$, where $1\le k\le 3$.} By properties~(a) and~(b) and assertion~(3) of Proposition~\ref{propos_small_groups} we obtain that $G$ is isomorphic to one of the groups~$\rC_3$,  $\rC_3^2$, $\rC_9$, $\rC_3^3$, $\He_3$, or~$3_-^{1+2}$. Moreover, by property~(d) the group~$G$ acts freely on~$V$. All these actions are present in Table~\ref{table_main}.
\smallskip

\textsl{Case 3: $|G|=2^k\cdot 3$, where $1\le k\le 3$.} Since $G$ contains no elements of order~$6$, from assertion~(4) of Proposition~\ref{propos_small_groups} it follows that $G$ is isomorphic to one of the three groups~$\rS_3$, $\rA_4$, or~$\rS_4$. By property~(d) the lengths of all $G$-orbits in~$V$ are divisible by~$3$. On the other hand, the lengths of all $G$-orbits must divide~$|G|$. So the only possible orbit lengths are~$3$, $6$, $12$, or~$24$.

If $G\cong \rS_3$, then every $G$-orbit in~$V$ has length either~$3$ or~$6$. Any element of order~$2$ in~$G$ acts without fixed points on every $G$-orbit of length~$6$ and fixes  exactly one point in every $G$-orbit of length~$3$. From Proposition~\ref{propos_2_4} it follows that there are exactly three $G$-orbits of length~$3$ and so exactly three $G$-orbits of length~$6$. This action is present in Table~\ref{table_main}.

If $G\cong\rA_4$, then every $G$-orbit in~$V$ has length $3$, $6$, or~$12$. Since $|V|$ is odd, at least one $G$-orbit has length~$3$. Suppose that $g\in G$ is an element of order~$2$. The action of~$G$ on each $G$-orbit~$\alpha$ of length~$3$ corresponds to the surjective homomorphism $\rA_4\to\rC_3$, so $g$ fixes every point of~$\alpha$.
Moreover, if $V$ contained a $G$-orbit~$\beta$ of length~$6$, then the action on~$\beta$ would be isomorphic to the left action of~$\rA_4$ on left cosets of~$\rC_2$ in~$\rA_4$. Hence, $g$ would fix exactly two points of~$\beta$. By Proposition~\ref{propos_2_4} the element $g$ fixes exactly three points in~$V$. Therefore, $G$ acts on~$V$ with one orbit of length~$3$ and two orbits of length~$12$. This action is present in Table~\ref{table_main}.

Finally, suppose that $G\cong\rS_4$. Since $|V|$ is odd, at least one $G$-orbit~$\alpha$ has length~$3$. The action of~$G$ on~$\alpha$ corresponds to the surjective homomorphism $\rS_4\to\rS_3$, so any element $g\in G$ of order~$4$ swaps a pair of points of~$\alpha$. It follows that the number of fixed points of~$g^2$ in~$V$ is strictly greater that the number of fixed points of~$g$ in~$V$, which contradicts Proposition~\ref{propos_2_4}. So this case is impossible.
\smallskip

\textsl{Case 4: $|G|=2^k\cdot 9$, where $1\le k\le 3$.} Since $G$ contains no elements of orders~$6$ or~$8$,  from assertions~(5)--(7) of Proposition~\ref{propos_small_groups} it follows that $G$ is one of the four groups~$\rD_9$, $\rC_3\rtimes\rS_3$,  $\rC_3^2\rtimes\rC_4$ (with $\rC_4$ acting faithfully on~$\rC_3^2$), or~$\PSU(3,2)$.
From property~(d) it follows that $G$ acts on~$V$ either with three orbits of length~$9$ or with two orbits of lengths~$18$ and~$9$, respectively. Consider the two subcases.
\smallskip

\textsl{Subcase 4a: $G$ acts on~$V$ with three orbits of length~$9$.} Since all Sylow $2$-subgroups of~$G$ are conjugate to each other, $G$ has a unique up to isomorphism transitive action on $9$ points.  The corresponding actions of $\rD_9$, $\rC_3\rtimes\rS_3$, and $\rC_3^2\rtimes\rC_4$ on $27$ points with three orbits of length~$9$ are present in Table~\ref{table_main}, and the action of~$\PSU(3,2)$ is mentioned separately in Proposition~\ref{propos_group}.
\smallskip

\textsl{Subcase 4b: $G$ acts on~$V$ with two orbits~$\alpha$ and~$\beta$ of lengths~$18$ and~$9$, respectively.}
First, assume that $G$ is isomorphic to  either~$\rD_9$ or~$\rC_3\rtimes\rS_3$. It can be checked directly that any element $g\in G$ of order~$2$ fixes exactly one point in~$\beta$. Moreover, $G$ acts regularly transitively on~$\alpha$. Hence, $g$ fixes exactly one point in~$V$, which contradicts Proposition~\ref{propos_2_4}.

Second, assume that $G$ is isomorphic to either~$\rC_3^2\rtimes\rC_4$ or~$\PSU(3,2)$. It is easy to check that
\begin{itemize}
\item all two-element subgroups of~$\rC_3^2\rtimes\rC_4$ are conjugate to each other,
\item four-element subgroups of~$\PSU(3,2)$ form three conjugacy classes, which can be taken to each other by outer automorphisms of~$\PSU(3,2)$.
\end{itemize}
So each of the two groups~$G$ has a unique up to isomorphism transitive action on $18$ points, which is thus realized on~$\alpha$. One can easily check that in both cases the group~$G$ contains  an element~$g$ of order~$4$ that acts on~$\alpha$ with an orbit of length~$2$. Therefore, the number of fixed points of~$g^2$ is strictly greater than the number of fixed points of~$g$, which again contradicts Proposition~\ref{propos_2_4}. So this subcase is impossible.
\smallskip

\textsl{Case 5: $|G|$ is either~$11$ or~$22$.} Since $G$ contains no elements of order~$22$, it follows from Proposition~\ref{propos_2p} that $G$ is isomorphic to either~$\rC_{11}$ or~$\rD_{11}$. If $G\cong\rC_{11}$, then by property~(e) it acts on~$V$ with two orbits of length~$11$ and five fixed points. If $G\cong\rD_{11}$, then by property~(f) it acts on~$V$ with two orbits of length~$11$, two orbits of length~$2$, and one fixed point. Both these actions are present in Table~\ref{table_main}.
\smallskip

\textsl{Case 6: $|G|=2^k\cdot 13$, where $0\le k\le 2$.} Since $G$ contains no elements of order~$26$, it follows from Proposition~\ref{propos_2p} and assertion~(8) of Proposition~\ref{propos_small_groups}  that $G$ is isomorphic to one of the three groups~$\rC_{13}$, $\rD_{13}$, or~$\rC_{13}\rtimes\rC_4$ (with $\rC_4$ acting faithfully on~$\rC_{13}$).  Let $H$ be a Sylow $13$-subgroup of~$G$; then $H\cong\rC_{13}$. By property~(g) the subgroup~$H$ acts on~$V$ with three orbits of lengths~$13$, $13$, and~$1$. Every $G$-orbit is the union of several $H$-orbits, and the length of every $G$-orbit divides~$|G|$. Therefore, the lengths of $G$-orbits in~$V$ are either $13$, $13$, and~$1$ or $26$ and~$1$. Consider the two subcases.
\smallskip

\textsl{Subcase 6a: $G$ acts on~$V$ with three orbits of lengths~$13$, $13$, and~$1$.}  Each of the three possible groups~$G$ has a unique up to isomorphism transitive action on $13$ points. The corresponding three actions  on $27$ points with orbit lengths~$13$, $13$, and~$1$  are present in Table~\ref{table_main}.
\smallskip

\textsl{Subcase 6b: $G$ acts on~$V$ with two orbits of lengths~$26$ and~$1$.} Then $G$ is isomorphic to either~$\rD_{13}$ or~$\rC_{13}\rtimes\rC_4$. If $G\cong\rD_{13}$, then $G$ acts regularly transitively on the orbit of length~$26$, so every element $g\in G$ of order~$2$  fixes only one point in~$V$, which  contradicts Proposition~\ref{propos_2_4}. If $G\cong\rC_{13}\rtimes\rC_4$, then $G$ contains an element~$g$ of order~$4$. The stabilizer~$G_v$ of any point~$v$ in the orbit of length~$26$ has order~$2$; hence, $g\notin G_v$. Therefore, $g$ fixes only one point in~$V$, which again contradicts Proposition~\ref{propos_2_4}. So this subcase is impossible.
\smallskip

\textsl{Case 7: $|G|=39$.}
From property~(d) it follows that the length of every $G$-orbit in~$V$ is divisible by~$3$. Since the length of every $G$-orbit divides~$39$, it follows that all $G$-orbits in~$V$ have length~$3$. We arrive at a contradiction, since $G$ contains an element of order~$13$. So this case is impossible.

\smallskip

\textsl{Case 8: $|G|=54$.} We are going to prove that this case is also impossible. Let $N$ be the number of pairs $(g,v)\in G\times V$ such that $g$ has order~$2$ and $g(v)=v$. To arrive at a contradiction, we compute $N$ in two ways and show that the results are different.

Firstly, from property~(d) it follows that $G$ acts transitively on~$V$. Hence the stabilizer~$G_v$ of every point $v\in V$ has order~$2$. Therefore $N=|V|=27$.

Secondly, from property~(a) it follows that  the order of every element of~$G$ is either~$2$ or a power of~$3$. The Sylow $3$-subgroup~$H$ of~$G$ has index~$2$ and hence is normal. It follows that any $3$-subgroup of~$G$ is contained in~$H$, so $G$ has exactly $26$ elements $g\ne 1$ whose order is a power of~$3$. Therefore, $G$ contains exactly $27$ elements $g$ that have order~$2$. By Proposition~\ref{propos_2_4} each such element fixes exactly three points in~$V$. Thus $N=3\cdot 27=81$. A contradiction.
\smallskip

\textsl{Case 9: $|G|=351$.}
Since $G$ contains no elements of order~$39$, from assertion~(9) of Proposition~\ref{propos_small_groups} it follows that $G$ is isomorphic to the semi-direct product~$\rC_3^3\rtimes\rC_{13}$. It is easy to check that all subgroups of order~$13$ of~$\rC_3^3\rtimes\rC_{13}$ are conjugate to each other, so this group admits exactly one (up to isomorphism) action on $27$~points. This action is in Table~\ref{table_main}.
\smallskip

We have looked through all the cases. The proposition follows.
\end{proof}

\section{Absence of triangulations with symmetry group~$\PSU(3,2)$}\label{section_no_PSU}

In this section we prove Proposition~\ref{propos_no_PSU}. We assume that $K$ is a  $27$-vertex $\Z$-homology $16$-manifold such that $H_*(K;\Z)\not\cong H_*(S^{16};\Z)$ and $\Sym(K)\cong\PSU(3,2)$, and our goal is to arrive at a contradiction. We denote the group~$\Sym(K)$ by~$G$.

Throughout this section, a special role will be played by subgroups $H\subset G$ that are isomorphic to~$\rC_4$. We denote the set of all such subgroups~$H$ by~$\CH$. By Remark~\ref{remark_PSU} a Sylow $2$-subgroup of~$G$ is isomorphic to~$\rQ_8$, so $\CH$ is non-empty. In fact, it is easy to check that the set~$\CH$ consists of $3$ conjugacy classes of $9$ subgroups each, but we will not need this fact.

Recall that from Theorems~\ref{thm_Novik} and~\ref{thm_ArMa} it follows that $K$ satisifies complementarity. Then by Proposition~\ref{propos_comp_neigh} we obtain that $K$ is $9$-neighborly.

Let $V$ denote the vertex set of~$K$. By Propositions~\ref{propos_prop} and~\ref{propos_group} the group~$G$ acts on~$V$ with three orbits of length~$9$. We denote these orbits by~$\alpha_0$, $\alpha_1$, and~$\alpha_2$, where we conveniently consider the indices~$0$, $1$, and~$2$ as elements of the field~$\F_3$. Since $K$ is $9$-neighborly, we have that $\alpha_0$, $\alpha_1$, and~$\alpha_2$ are simplices of~$K$. By Proposition~\ref{propos_2_4} each subgroup $H\in\CH$  acts on~$V$ with six orbits of length~$4$ and three fixed points. Then every $G$-orbit~$\alpha_t$ contains one $H$-fixed point and two $H$-orbits of length~$4$.

We identify the set of vertices of~$\CP^2_9$ with an affine plane~$\mathcal{P}$ over~$\F_3$ as in Subsection~\ref{subsection_K}.

\begin{propos}\label{propos_PSU_aux}
The following assertions hold (possibly, after reversing the cyclic order of the orbits~$\alpha_0$, $\alpha_1$, and~$\alpha_2$):
\begin{enumerate}
 \item
 $
 \link(\alpha_t,K)=\partial\alpha_{t+1}
 $ for all $t\in\F_3$.
 \item For each subgroup $H\in\CH$, we have an isomorphism $$\varphi_H\colon K^H\to\CP^2_9$$ and we can choose affine coordinates~$(x,y)$ in the vertex set~$\mathcal{P}$ of~$\CP^2_9$ so that
 \begin{itemize}
  \item the special lines in~$\mathcal{P}$ are $\{y=0\}$, $\{y=1\}$, and $\{y=2\}$ in this cyclic order,
  \item the bijection~$V/H\to\mathcal{P}$ induced by~$\varphi_H$ is given by Table~\ref{table_bijection_id}, where $u_t\in\alpha_t$ is an $H$-fixed point and~$\beta_t$ and~$\gamma_t$ are two $H$-orbits of length~$4$ contained in~$\alpha_t$.
 \end{itemize}
 \item  Moreover, for all $t\in\F_3$,
 \begin{gather}\label{eq_cond1}
  \{u_t\}\cup\bigl(\alpha_{t+1}\setminus\{u_{t+1}\}\bigr)\cup\bigl(\alpha_{t+2}\setminus\{u_{t+2}\}\bigr)\in K,\\
  \label{eq_cond2}
  \{u_0,u_1,u_2\}\cup\beta_{t+1}\cup\beta_{t+2}\notin K,\\
  \label{eq_cond3}
  \{u_0,u_1,u_2\}\cup\gamma_{t+1}\cup\gamma_{t+2}\notin K.
 \end{gather}
\end{enumerate}
\end{propos}

\begin{table}
 \caption{The bijection~$\varphi_H$ between points $(x,y)\in\mathcal{P}$ and $H$-orbits in~$V$}\label{table_bijection_id}
 \begin{tabular}{|c|ccc|}
 \hline
 \diaghead(1,-1){abc}{$y$}{$x$} & $0$ & $1$ & $2$\\
 \hline
 $0$ & $u_0\vphantom{\Bigl(}$ & $\beta_0$  & $\gamma_0$\\
 $1$ & $u_1\vphantom{\Bigl(}$ & $\beta_1$  & $\gamma_1$\\
 $2$ & $u_2\vphantom{\Bigl(}$ & $\beta_2$  & $\gamma_2$\\
 \hline
 \end{tabular}
\end{table}

\begin{remark}
 We conveniently think that the isomorphisms~$\varphi_H$ for different~$H$ lead to different copies of~$\CP^2_9$, so the coordinates $(x,y)$ for different~$H$ are not related with each other. Certainly, the fixed vertices~$u_t$ and the orbits~$\beta_t$ and~$\gamma_t$ depend on the choice of~$H$. For the sake of simplicity, we do not reflect this dependence in the notation.
\end{remark}

\begin{proof}[Proof of Proposition~\ref{propos_PSU_aux}]
 Suppose that $H\in\CH$. For each $t\in\F_3$, the $G$-orbit~$\alpha_t$ consists of the $H$-fixed point~$u_t$ and two $H$-orbits~$\beta_t$ and~$\gamma_t$ of length~$4$:
 $$
 \alpha_t=\{u_t\}\cup\beta_t\cup\gamma_t.
 $$
Note that we have the freedom to choose which of the two orbits is~$\beta_t$ and which is~$\gamma_t$. We will take the opportunity to swap these orbits later.

 By Proposition~\ref{propos_2} we have an isomorphism  $\varphi_H\colon K^H\to\CP^2_9$.  The isomorphism~$\varphi_H$ provides a bijection $V/H\to\mathcal{P}$, which we also denote by~$\varphi_H$. Under this bijection some $6$ points in~$\mathcal{P}$ correspond to $H$-orbits of length~$4$, and the other $3$ points in~$\mathcal{P}$ correspond to the $H$-fixed vertices of~$K$. We denote by $m\subset  \mathcal{P}$ the $3$-element subset consisting of the points corresponding to the $H$-fixed vertices of~$K$. Then every $4$-element subset~$\sigma\subset \mathcal{P}$ that contains~$m$ corresponds to a $7$-element subset of~$V$, which is a simplex of~$K$, since $K$ is $9$-neighborly. Hence $\sigma$ is a simplex of~$\CP^2_9$. From Proposition~\ref{propos_3subset} it follows that $m$ is a non-special line. Hence, we can choose affine coordinates~$(x,y)$ in~$\mathcal{P}$ so that
\begin{itemize}
 \item $\varphi_H(u_0)=(0,0)$.
 \item $m$ coincides with the line $\{x=0\}$,
 \item the special lines in~$\mathcal{P}$ are $\{y=0\}$, $\{y=1\}$, and $\{y=2\}$ in this cyclic order.
\end{itemize}
Swapping if necessary the orbits~$\alpha_1$ and~$\alpha_2$, we can achieve that
$$
\varphi_H(u_t)=(0,t),\qquad t\in\F_3.
$$

We denote the stabilizer~$G_{u_0}$ by~$S$; then $H\subset S$. By Remark~\ref{remark_PSU}, we have $S\cong\rQ_8$. By the Sylow theorems all subgroups of~$G$ of order~$8$ are conjugate to each other. So every $G$-orbit~$\alpha_t$ contains a point with stabilizer~$S$. This point is~$u_t$, since it is the only point that is fixed by~$H$. So we have that $G_{u_t}=S$ for all $t\in\F_3$. Since  $S\cong\rQ_8$, we see that $H$ is a normal subgroup of~$S$.
The action of~$S$ on~$K$ induces an action of the quotient group~$S/H\cong\rC_2$ on~$K^H$, which we identify with~$\CP^2_9$ along~$\varphi_H$. We denote by~$s$ the involutive automorphism of~$\CP^2_9$ given by the generator of~$S/H$. From Proposition~\ref{propos_8} it follows that the group~$S$ acts on~$V$ with three fixed points~$u_t$ and three orbits~$\beta_t\cup\gamma_t$ of length~$8$. Hence, for each $t\in\F_3$, the automorphism~$s$ fixes the  point~$\varphi_H(u_t)=(0,t)$ and swaps the points~$\varphi_H(\beta_t)$ and~$\varphi_H(\gamma_t)$. On the other hand, since $s\in\Sym(\CP^2_9)$, we see that $s$ is an affine transformation of~$\mathcal{P}$ that  takes special lines to special lines preserving their cyclic order.
It follows easily that
$$
s(x,y)=(-x,y)
$$
for all $(x,y)\in\mathcal{P}$.
Therefore, for each $t\in\F_3$, the points~$\varphi_H(\beta_t)$ and~$\varphi_H(\gamma_t)$ have the same $y$-coordinate~$\nu(t)$. Then $\nu$ is a permutation of elements of~$\F_3$. Besides, swapping the orbits in pairs~$\{\beta_t,\gamma_t\}$, we can achieve that
\begin{align*}
\varphi_H(\beta_t)=\bigl(1,\nu(t)\bigr), \qquad \varphi_H(\gamma_t)=\bigl(2,\nu(t)\bigr)
\end{align*}
for all $t\in\F_3$

We would like to prove that the permutation~$\nu$ is in fact trivial. To do this, we enumerate all non-trivial permutations and arrive at a contradiction in each case. There are three cases:

\smallskip

\textsl{Case 1: $\nu$ is a transposition.} We may assume that $\nu=(1\ 2)$, since the cases of the two other transpositions are similar. Then the bijection $\varphi_H\colon V/H\to\mathcal{P}$ is as in Table~\ref{table_bijection_wrong}(a). By the construction  of~$\CP^2_9$ we have that
the following two sets are simplices of~$\CP^2_9$:
\begin{equation}\label{eq_2simplices}
\begin{aligned}
 &\bigl\{(0,1),(1,0),(2,0),(1,2),(2,2)\bigr\},\\
 &\bigl\{(0,1),(1,1),(2,1),(1,2),(2,2)\bigr\}.
\end{aligned}
\end{equation}
Hence, the corresponding unions of $H$-orbits are simplices of~$K$. These unions of $H$-orbits are
\begin{align*}
 \{u_1\}\cup\beta_0\cup\gamma_0\cup\beta_1\cup\gamma_1&=\alpha_1\cup\bigl(\alpha_0\setminus\{u_0\}\bigr),\\
 \{u_1\}\cup\beta_2\cup\gamma_2\cup\beta_1\cup\gamma_1&=\alpha_1\cup\bigl(\alpha_2\setminus\{u_2\}\bigr),
\end{align*}
respectively. Therefore, $\alpha_0\setminus\{u_0\}$ and $\alpha_2\setminus\{u_2\}$ are simplices of~$\link(\alpha_1,K)$. But the simplex~$\alpha_1$ is $G$-invariant, so the simplicial complex~$\link(\alpha_1,K)$ is also $G$-invariant. Since $G$ acts transitively on each of the sets~$\alpha_0$ and~$\alpha_2$, it follows that $\link(\alpha_1,K)$ contains the simplices~$\alpha_0\setminus\{v\}$ for all~$v\in\alpha_0$ and the simplices~$\alpha_2\setminus\{v\}$ for all~$v\in\alpha_2$. In other words, $\link(\alpha_1,K)$ contains simultaneously the boundaries of the $8$-simplices~$\alpha_0$ and~$\alpha_2$. This yields a contradiction, since $K$ is a $\Z$-homology $16$-manifold and hence $\link(\alpha_1,K)$ must be a connected $\Z$-homology $7$-manifold with the homology of a $7$-sphere.

\begin{table}
 \caption{Wrong bijections~$\varphi_H$}\label{table_bijection_wrong}
 \begin{tabular}{|c|ccc|c|c|ccc|c|c|ccc|}
 \multicolumn{4}{c}{(a) Case~$1\vphantom{\Bigl(}$} & \multicolumn{1}{c}{\hphantom{abcde}} &
 \multicolumn{4}{c}{(b) Case~$2\vphantom{\Bigl(}$} & \multicolumn{1}{c}{\hphantom{abcde}} &
 \multicolumn{4}{c}{(c) Case~$3\vphantom{\Bigl(}$}\\
 \cline{1-4}\cline{6-9}\cline{11-14}
 \diaghead(1,-1){abc}{$y$}{$x$} & $0$ & $1$ & $2$
 & &  \diaghead(1,-1){abc}{$y$}{$x$} & $0$ & $1$ & $2$
 & &  \diaghead(1,-1){abc}{$y$}{$x$} & $0$ & $1$ & $2$
 \\
 \cline{1-4}\cline{6-9}\cline{11-14}
 $0$ & $u_0\vphantom{\Bigl(}$ & $\beta_0$  & $\gamma_0$ &&
 $0$ & $u_0\vphantom{\Bigl(}$ & $\beta_2$  & $\gamma_2$ &&
 $0$ & $u_0\vphantom{\Bigl(}$ & $\beta_1$  & $\gamma_1$
 \\
 $1$ & $u_1\vphantom{\Bigl(}$ & $\beta_2$  & $\gamma_2$ &&
 $1$ & $u_1\vphantom{\Bigl(}$ & $\beta_0$  & $\gamma_0$ &&
 $1$ & $u_1\vphantom{\Bigl(}$ & $\beta_2$  & $\gamma_2$
 \\
 $2$ & $u_2\vphantom{\Bigl(}$ & $\beta_1$  & $\gamma_1$ &&
 $2$ & $u_2\vphantom{\Bigl(}$ & $\beta_1$  & $\gamma_1$ &&
 $2$ & $u_2\vphantom{\Bigl(}$ & $\beta_0$  & $\gamma_0$\\
 \cline{1-4}\cline{6-9}\cline{11-14}
 \end{tabular}
\end{table}

\smallskip

\textsl{Case 2: $\nu(t)=t+1$ for all $t\in\F_3$.} Then the bijection $\varphi_H\colon V/H\to\mathcal{P}$ is as in Table~\ref{table_bijection_wrong}(b). We arrive at a contradiction literally as in the previous case. Indeed, again sets~\eqref{eq_2simplices} are simplices of~$\CP^2_9$. Therefore, $\alpha_0\setminus\{u_0\}$ and~$\alpha_2\setminus\{u_2\}$ are simplices of~$\link(\alpha_1,K)$ and hence $\partial\alpha_0$ and~$\partial\alpha_2$ are subcomplexes of~$\link(\alpha_1,K)$, which is impossible.

\smallskip

\textsl{Case 3: $\nu(t)=t-1$ for all $t\in\F_3$.} Then the bijection $\varphi_H\colon V/H\to\mathcal{P}$ is as in Table~\ref{table_bijection_wrong}(c). This case is a little different. We have that $$\bigl\{(0,1),(1,0),(2,0),(1,2),(2,2)\bigr\}$$
is a simplex of~$\CP^2_9$. Hence,
$$
\{u_1\}\cup\beta_1\cup\gamma_1\cup\beta_0\cup\gamma_0=\alpha_1\cup\bigl(\alpha_0\setminus\{u_0\}\bigr)
$$
is a simplex of~$K$. As in Case~1, this implies that $\link(\alpha_1,K)$ contains a subcomplex~$\partial\alpha_0$. Since $\link(\alpha_1,K)$ must be a connected $\Z$-homology $7$-manifold, it follows that in fact $\link(\alpha_1,K)=\partial\alpha_0$. On the other hand, we have that
$$
\bigl\{(0,1),(1,0),(1,1),(1,2),(2,0)\bigr\}
$$
is  a simplex of~$\CP^2_9$ and hence
$$
\{u_1\}\cup\beta_0\cup\beta_1\cup\beta_2\cup\gamma_1=\alpha_1\cup\bigl(\beta_0\cup\beta_2\bigr)
$$
is a simplex of~$K$. Therefore, $\beta_0\cup\beta_2$ is a simplex of~$\link(\alpha_1,K)$, and we arrive at a contradiction, since $\beta_0\cup\beta_2$ is not contained in~$\alpha_0$.

\smallskip

Thus,  $\nu$ is a trivial permutation. Hence, the bijection $\varphi_H\colon V/H\to\mathcal{P}$ is as in Table~\ref{table_bijection_id}, and so we obtain assertion~(2).

For each~$t\in\F_3$, we have that
\begin{align}\label{eq_simp1}
&\bigl\{(0,t),(1,t),(2,t),(1,t+1),(2,t+1)\bigr\},\\
\label{eq_simp2}
&\bigl\{(0,t),(1,t+1),(2,t+1),(1,t+2),(2,t+2)\bigr\}
\end{align}
are simplices of~$\CP^2_9$ and
\begin{align}
\label{eq_nonsimp1}
&\bigl\{(0,0),(0,1),(0,2),(1,t+1),(1,t+2)\bigr\},\\
\label{eq_nonsimp2}
&\bigl\{(0,0),(0,1),(0,2),(2,t+1),(2,t+2)\bigr\}
\end{align}
are not simplices of~$\CP^2_9$. So the unions of the $H$-orbits corresponding to the elements of sets~\eqref{eq_simp1} and~\eqref{eq_simp2} are simplices of~$K$ and the unions of the $H$-orbits corresponding to the elements of sets~\eqref{eq_nonsimp1} and~\eqref{eq_nonsimp2} are not simplices of~$K$. For~\eqref{eq_simp2}, \eqref{eq_nonsimp1}, and~\eqref{eq_nonsimp2}, we obtain exactly the statements~\eqref{eq_cond1}, \eqref{eq_cond2}, and~\eqref{eq_cond3} from assertion~(3), respectively. For~\eqref{eq_simp1}, we obtain that
\begin{equation*}
\alpha_t\cup(\alpha_{t+1}\setminus\{u_{t+1}\})\in K.
\end{equation*}
As above, since $G$ acts transitively on~$\alpha_{t+1}$, it follows that $\link(\alpha_t,K)\supseteq \partial\alpha_{t+1}$. Since $\link(\alpha_1,K)$ is a connected $\Z$-homology $7$-manifold, we obtain that $\link(\alpha_t,K)=\partial\alpha_{t+1}$, which is assertion~(1).
\end{proof}

\begin{remark}
 Though to prove assertion~(1) we have considered a subgroup $H\in\CH$, the resulting cyclic order of the orbits~$\alpha_0$, $\alpha_1$, and~$\alpha_2$ is certainly independent of~$H$, since the assertion cannot be true simultaneously for the two opposite orders. Further we always number the orbits~$\alpha_0$, $\alpha_1$, and~$\alpha_2$ so that the assertions of Proposition~\ref{propos_PSU_aux} are true for all~$H$.
\end{remark}

The group~$G$ has $9$ Sylow $2$-subgroups, each isomorphic to~$\rQ_8$, see Remark~\ref{remark_PSU}. By the Sylow theorems all of them are conjugate to each other. Hence, for each of the three orbits~$\alpha_0$, $\alpha_1$, and~$\alpha_2$, the stabilizers of the $9$ vertices in this orbit are exactly the $9$ different subgroups of~$G$ isomorphic to~$\rQ_8$. Therefore, there is a unique permutation~$\theta$ of the set~$V$ such that:
\begin{itemize}
 \item $\theta(\alpha_t)=\alpha_{t+1}$ for all $t\in\F_3$,
 \item $G_v=G_{\theta(v)}$ for all~$v\in~V$.
\end{itemize}
 Obviously, $\theta^3=1$. It follows immediately from the construction that the permutation~$\theta$ commutes with the action of~$G$ on~$V$. So, for each subgroup $H\in\CH$, the permutation~$\theta$ takes $H$-orbits to $H$-orbits. Let us use the notation for the $H$-orbits from Proposition~\ref{propos_PSU_aux}. Then
 $$
 \theta(u_t)=u_{t+1},\qquad t\in\F_3.
 $$

 \begin{propos}\label{propos_cor_PSU_aux}
For each $t\in\F_3$ and each vertex $v\in \alpha_t$,  the set
 \begin{equation*}
  \{v\}\cup\bigl(\alpha_{t+1}\setminus\{\theta(v)\}\bigr)\cup\bigl(\alpha_{t+2}\setminus\{\theta^2(v)\}\bigr)
 \end{equation*}
 is a simplex of~$K$.
\end{propos}

\begin{proof}
 The stabilizer~$G_v$ of~$v$ is isomorphic to~$\rQ_8$ and hence contains a subgroup $H\in\CH$ (in fact, exactly three such subgroups). Then $v$ is the $H$-fixed point in~$\alpha_t$, that is, the point~$u_t$ for the subgroup~$H$. The required assertion now follows from assertion~(3) of Proposition~\ref{propos_PSU_aux}.
\end{proof}

For each subgroup $H\in\CH$ and each $t\in\F_3$, the permutation~$\theta$ takes the pair of $H$-orbits~$\{\beta_{t+1},\gamma_{t+1}\}$ to the pair of $H$-orbits~$\{\beta_{t+2},\gamma_{t+2}\}$, where we use the notation for $H$-orbits from Proposition~\ref{propos_PSU_aux}. Nevertheless, it is by no means true that $\theta$ always takes~$\beta_{t+1}$ to~$\beta_{t+2}$ and $\gamma_{t+1}$ to~$\gamma_{t+2}$. We conveniently introduce signs $\eta_{H,t}=\pm 1$ indexed by pairs $(H,t)\in\CH\times\F_3$ in the following way:
 $$
   \eta_{H,t}=
   \left\{\begin{aligned}
   &1&&\text{if}\ \theta(\beta_{t+1})=\beta_{t+2}\ \text{and}\ \theta(\gamma_{t+1})=\gamma_{t+2},\\
   -&1&&\text{if}\ \theta(\beta_{t+1})=\gamma_{t+2}\ \text{and}\ \theta(\gamma_{t+1})=\beta_{t+2}.
   \end{aligned}\right.
 $$
Since $\theta^3=1$, we obtain the following statement.

 \begin{propos}\label{propos_eta}
  $\eta_{H,0}\eta_{H,1}\eta_{H,2}=1$ for all $H\in\CH$.
 \end{propos}

 \begin{remark}\label{remark_beta_gamma}
  Actually, the orbits~$\beta_t$ and~$\gamma_t$ are not determined uniquely by the conditions from assertion~(2) of Propostion~\ref{propos_PSU_aux}. Namely, we have the freedom to change the sign of the coordinate~$x$ on the plane~$\mathcal{P}$ and swap simultaneously the orbits in every pair~$\{\beta_t,\gamma_t\}$ with $t\in\F_3$. Nevertheless, this operation does not change the signs~$\eta_{H,t}$, so they are well defined.
 \end{remark}

 Furher, if a group~$H$ acts on a set~$Z$ and $z_1,z_2\in Z$, we put
 $$
\varepsilon_{H}(z_1,z_2)=\left\{
\begin{aligned}
 &1&&\text{if $z_1$ and $z_2$ lie in the same $H$-orbit,}\\
 -&1&&\text{if $z_1$ and $z_2$ lie in different $H$-orbits.}
\end{aligned}
\right.
$$
The definitions of $\eta_{H,t}$ and~$\varepsilon_H(z_1,z_2)$ immediately imply the following assertion.

\begin{propos}\label{propos_bb_gg}
Suppose that $H\in\CH$. Denote the $H$-orbits in~$V$ as in Proposition~\ref{propos_PSU_aux}. Suppose that $t\in\F_3$, $v\in\alpha_{t+1}\setminus\{u_{t+1}\}$ and $w\in\alpha_{t+2}\setminus\{u_{t+2}\}$. Then
\begin{itemize}
 \item if $\varepsilon_H\bigl(\theta(v),w\bigr)=\eta_{H,t}$, then either $v\in\beta_{t+1}$ and~$w\in\beta_{t+2}$ or $v\in\gamma_{t+1}$ and $w\in\gamma_{t+2}$,
 \item if $\varepsilon_H\bigl(\theta(v),w\bigr)=-\eta_{H,t}$, then either $v\in\beta_{t+1}$ and~$w\in\gamma_{t+2}$ or $v\in\gamma_{t+1}$ and $w\in\beta_{t+2}$.
\end{itemize}
\end{propos}

We will also need one easy observation concerning actions of the quaternion group~$\rQ_8$.

\begin{lem}\label{lem_eps}
Suppose that the group $\rQ_8$ acts transitively and freely on a $8$-element set~$Z$. Let~$H_1$, $H_2$, and~$H_3$ be the three subgroups of~$\rQ_8$ that are isomorphic to~$\rC_4$. Then
$$\varepsilon_{H_1}(z_1,z_2)\varepsilon_{H_2}(z_1,z_2)\varepsilon_{H_3}(z_1,z_2)=1$$ for all $z_1,z_2\in Z$.
\end{lem}
\begin{proof}
 Let $g\in \rQ_8$ be the element that takes~$z_1$ to~$z_2$. The assertion of the lemma follows immediately from the fact that $g$ belongs either to all the three subgroups~$H_1$, $H_2$, and~$H_3$ or to  exactly one of them.
\end{proof}

 Now, suppose that $t\in\F_3$, $v\in\alpha_{t+1}$ and~$w\in\alpha_{t+2}$. By assertion~(1) of Proposition~\ref{propos_PSU_aux} we have that $\sigma=\alpha_{t+1}\cup(\alpha_{t+2}\setminus\{w\})$ is a $16$-simplex of~$K$. Then
 $$
 \rho=\sigma\setminus\{v\}=(\alpha_{t+1}\setminus\{v\})\cup(\alpha_{t+2}\setminus\{w\})
 $$
 is a $15$-simplex of~$K$. Since $K$ is a $16$-pseudomanifold, we obtain that the simplex $\rho$ is contained in exactly one $16$-simplex $\tau\in K$ that is different from~$\sigma$. Since $\link(\alpha_{t+2},K)=\partial\alpha_t$, we see that $\tau\ne (\alpha_{t+1}\setminus\{v\})\cup\alpha_{t+2}$. So there is exactly one vertex $u\in\alpha_t$ such that
 $$
 \tau=(\alpha_{t+1}\setminus\{v\})\cup(\alpha_{t+2}\setminus\{w\})\cup\{u\}
 $$
 is a simplex of~$K$. We obtain a well-defined map
 \begin{align*}
 f_t\colon \alpha_{t+1}\times\alpha_{t+2}&\to\alpha_t,\\
 (v,w)&\mapsto u.
 \end{align*}
 Certainly, the map~$f_t$ is $G$-equivariant with respect to the diagonal action of~$G$ on the set~$\alpha_{t+1}\times\alpha_{t+2}$. Since $G$ acts transitively on~$\alpha_t$, we obtain that all pre-images $f_t^{-1}(u)$, where $u\in\alpha_t$, consist of the same number of elements. Hence $| f_t^{-1}(u)|=9$ for all~$u$.

 \begin{lem}\label{lem_pairs}
  \begin{enumerate}
   \item If $w=\theta(v)$, then $f_t(v,w)=\theta^2(v)$.
   \item If $w\ne\theta(v)$, then $f_t(v,w)$ is neither~$\theta^2(v)$ nor~$\theta(w)$.
  \end{enumerate}
 \end{lem}

 \begin{proof}
  Assertion~(1) is a reformulation of Proposition~\ref{propos_cor_PSU_aux}. Let us prove assertion~(2). Assume that $f_t(v,w)=\theta^2(v)$.  Then
   $$
    \tau=(\alpha_{t+1}\setminus\{v\})\cup(\alpha_{t+2}\setminus\{w\})\cup\{\theta^2(v)\}
   $$
  is a simplex of~$K$.
  Applying Proposition~\ref{propos_cor_PSU_aux} to the vertex~$v$, we obtain that the set
 $$
 \kappa=\{v\}\cup\bigl(\alpha_{t+2}\setminus\{\theta(v)\}\bigr)\cup\bigl(\alpha_{t}\setminus\{\theta^2(v)\}\bigr)
 $$
is a simplex of~$K$. Since $\tau\cup\kappa=V$, we arrive at a contradiction with complimentarity. Similarly, if $f_t(v,w)=\theta(w)$, then we arrive at a contradiciton by applying Proposition~\ref{propos_cor_PSU_aux} to the vertex~$w$.
 \end{proof}

\begin{propos}\label{propos_choose_H}
 There exists a subgroup $H\in\CH$, an element $t\in\F_3$, and vertices $v\in\alpha_{t+1}\setminus\{u_{t+1}\}$ and $w\in\alpha_{t+2}\setminus\{u_{t+2}\}$ such that
 \begin{enumerate}
  \item $f_t(v,w)=u_t$,
  \item $\varepsilon_{H}\bigl(\theta(v),w\bigr)=\eta_{H,t}$.
 \end{enumerate}
 (Here, as in Proposition~\ref{propos_PSU_aux}, we denote by $u_0$, $u_1$, and~$u_2$ the $H$-fixed points in~$\alpha_0$, $\alpha_1$, and~$\alpha_2$, respectively.)
\end{propos}

\begin{proof}
 Fix an arbitrary vertex~$u_0\in \alpha_0$ and put $u_1=\theta(u_0)$ and $u_2=\theta^2(u_0)$. By the definition of~$\theta$, we have $G_{u_0}=G_{u_1}=G_{u_2}$; we denote this stabilizer by~$S$. Then $S\cong\rQ_8$. Hence, $S$ contains exactly three subgroups that belong to~$\CH$. We denote these three subgroups by~$H_1$, $H_2$, and~$H_3$.

As was mentioned above, each pre-image~$f_t^{-1}(u_t)$ consists of $9$ elements. By Lemma~\ref{lem_pairs} one of these elements is  $(u_{t+1},u_{t+2})$, and the other $8$ elements $(v,w)\in f_t^{-1}(u_t)$ satisfy $v\ne u_{t+1}$ and $w\ne u_{t+2}$. Choose one of those $8$ pre-images and denote it by~$(v_{t+1},w_{t+2})$. Thus, we obtain $6$ vertices $v_0$, $v_1$, $v_2$, $w_0$, $w_1$, and~$w_2$ that have the following properties:
\begin{itemize}
 \item $v_t,w_t\in\alpha_t\setminus\{u_t\}$,
 \item $f_t(v_{t+1},w_{t+2})=u_t$.
\end{itemize}

For each $t\in\F_3$, the group~$S$ acts transitively and freely on the eight-element set $\alpha_{t+2}\setminus\{u_{t+2}\}$. By Lemma~\ref{lem_eps} we see that
$$
\varepsilon_{H_1}\bigl(\theta(v_{t+1}),w_{t+2}\bigr)\varepsilon_{H_2}\bigl(\theta(v_{t+1}),w_{t+2}\bigr)\varepsilon_{H_3}\bigl(\theta(v_{t+1}),w_{t+2}\bigr)=1,\qquad t\in\F_3.
$$
Taking the product of these three equations and the equations from Proposition~\ref{propos_eta} for the three subgroups~$H_1$, $H_2$, and~$H_3$, we obtain that
\begin{equation}\label{eq_prod}
\prod_{i=1}^3\prod_{t\in\F_3}\eta_{H_i,t}\varepsilon_{H_i}\bigl(\theta(v_{t+1}),w_{t+2}\bigr)=1.
\end{equation}
Therefore, at least one  multiplier $\eta_{H_i,t}\varepsilon_{H_i}\bigl(\theta(v_{t+1}),w_{t+2}\bigr)$ is equal to~$1$. Then the required assertions hold for the $4$-tuple $(H_i,t,v_{t+1},w_{t+2})$.
\end{proof}

\begin{proof}[End of the proof of Proposition~\ref{propos_no_PSU}]
 Let us show how to arrive at a contradiction using Propositions~\ref{propos_PSU_aux}, \ref{propos_bb_gg}, and~\ref{propos_choose_H}. Let $(H, t, v, w)$ be the $4$-tuple from Proposition~\ref{propos_choose_H}.  Denote the $H$-fixed points and the $H$-orbits in~$V$ as in Proposition~\ref{propos_PSU_aux}.

By assertion~(1) of Proposition~\ref{propos_choose_H}, we have that $f_t(v,w)=u_t$. Hence,
$$
\tau=\{u_t\}\cup\bigl(\alpha_{t+1}\setminus\{v\}\bigr)\cup\bigl(\alpha_{t+2}\setminus\{w\}\bigr)
$$
is a simplex of~$K$.

 From assertion~(2) of Proposition~\ref{propos_choose_H} and Proposition~\ref{propos_bb_gg} it follows that either $v\in\beta_{t+1}$ and~$w\in\beta_{t+2}$ or $v\in\gamma_{t+1}$ and~$w\in\gamma_{t+2}$. In the former case, $\tau$ contains the subset
 $$
 \mu=\{u_0,u_1,u_2\}\cup\gamma_{t+1}\cup\gamma_{t+2},
 $$
 and in the latter case, $\tau$ contains the subset
 $$
 \lambda=\{u_0,u_1,u_2\}\cup\beta_{t+1}\cup\beta_{t+2}.
 $$
 We arrive at a contradiction, since by assertion~(3) of Proposition~\ref{propos_PSU_aux} neither~$\lambda$ nor~$\mu$ is a simplex of~$K$. This completes the proof of Proposition~\ref{propos_no_PSU}.
\end{proof}

\end{document}